\documentclass[11pt,letterpaper]{amsart}
\usepackage{amssymb}
\usepackage{amsmath}
\usepackage{bbm}
\usepackage[dvips]{graphicx}
\usepackage{hyperref}
\usepackage[capitalize,nameinlink,noabbrev,compress]{cleveref}

\usepackage{todonotes}
\usepackage{enumitem}
\usepackage{comment}
\usepackage{mathdots}

\usepackage{tikz}
\usetikzlibrary{arrows}
\usetikzlibrary{decorations.markings}
\usetikzlibrary{shapes.geometric}

\numberwithin{equation}{section}

\newtheorem{thm}{Theorem}[section]
\newtheorem{prop}[thm]{Proposition}

\newtheorem{lem}[thm]{Lemma}
\newtheorem{defn}[thm]{Definition}
\newtheorem{rem}[thm]{Remark}
\newtheorem{conj}[thm]{Conjecture}

\newtheorem{claim}[thm]{Claim}

\newcommand{\ds}{\displaystyle}
\newcommand{\tsscpp}{totally symmetric self-complementary plane partition}

\title[Maximal path fluctuations in ASMs]
{GOE fluctuations for the maximum of the top path in alternating sign matrices}

\author{Arvind Ayyer}
\address{Arvind Ayyer, Department of Mathematics, 
Indian Institute of Science, Bangalore  560012, India.}
\email{arvind@iisc.ac.in}

\author{Sunil Chhita}
\address{Sunil Chhita, Department of Mathematical Sciences, 
Durham University, Durham, UK}
\email{sunil.chhita@durham.ac.uk}

\author{Kurt Johansson}
\address{Kurt Johansson, Department of Mathematics, 
KTH Royal Institute of Technology, Stockholm, Sweden.}
\email{kurtj@kth.se}


\date{\today}

\begin{document}

\begin{abstract}
The six-vertex model is an important toy-model in statistical mechanics for two-dimensional ice with a natural parameter $\Delta$. 
When $\Delta = 0$, the so-called free-fermion point, the model is in natural correspondence with domino tilings of the Aztec diamond.
Although this model is integrable for all $\Delta$, there has been very little progress in understanding its statistics in the scaling limit for other values.
In this work, we focus on the six-vertex model with domain wall boundary conditions at $\Delta = 1/2$, where it corresponds to alternating sign matrices (ASMs). We consider the level lines in a height function representation of ASMs.
We show that the maximum of the topmost level line for a uniformly random ASMs has the GOE Tracy--Widom distribution after appropriate rescaling.
A key ingredient in our proof is Zeilberger's proof of the ASM conjecture.
As far as we know, this is the first edge fluctuation result away from the tangency points for the domain-wall six-vertex model when we are not in the free fermion case.
\end{abstract}

\maketitle

\section{Introduction} 
\label{sec:Introduction}

The six-vertex model (also known as the square ice model) can be thought of as a random assignment of arrows to the edges in $\mathbb{Z}^2$ such that the number of incoming arrows equals the number of outgoing arrows at every vertex. Each such configuration is given a Boltzmann weight which is the product of local weights at every vertex. {Homogeneity and symmetry impose} conditions on the nature of these local weights and there are three independent parameters, commonly denoted $a, b, c$. The dependent parameter $\Delta = (a^2 + b^2 - c^2)/(2ab)$ plays an important role in the phase diagram. See the book by Baxter~\cite{Bax82} for more details. One of the most important finite versions of the six-vertex model is on an $n \times n$ subset of $\mathbb{Z}^2$ with extra half-edges on the boundaries where vertical half-edges point outwards and horizontal ones point inwards. This is known as the \emph{six-vertex model with domain wall boundary conditions} (DWBC). 
In this paper, we focus on the edge fluctuations for
this model of size $n$ with $a = b = c$ when $n$ tends to infinity.

Random tiling models, and in particular random domino tilings of the Aztec diamond,  give a good blueprint for the possible asymptotic behavior seen in the six vertex model with domain wall boundary conditions.  {Over the past couple of decades}, there has been much progress in understanding this particular random tiling model as well as other random tiling models. In general, a  limiting curve emerges which separates the random tiling into three possible macroscopic regions: \emph{frozen}, where the tiling is deterministic; \emph{rough}, where the correlations between tiles decay polynomially; and \emph{smooth}, where the correlations decay exponentially.  {The limiting curve, known as the
{\emph{arctic curve}}
can be obtained from the limiting height function, which is the solution of a variational principle \cite{CKP01}. There are} general methods to quantify this solution  depending on the type of random tiling model and its boundary conditions \cite{KO07, Duse}.   

Edge fluctuation results for random tilings of Aztec diamonds at the interface between the frozen and rough regions, where  the Airy-2 process governs the fluctuations of the interface after rescaling, have been obtained in \cite{Jo03}. The interface between the rough-smooth regions has been investigated in the two-periodic Aztec diamond where it is expected that  the Airy-2 process governs the interface after rescaling, see \cite{CJ16, DK21, BCJ20, BD19}. These results are not unique to domino tilings of the Aztec diamond and, except the fluctuations at the rough-smooth boundary, have been established for lozenge tilings beginning with \cite{OR01, FS03}. See the book~\cite{Gor20} for more information and references.

It is expected that the six-vertex model with domain wall boundary conditions shares similar asymptotic behaviors as random domino tilings of the Aztec diamond as well as other random tiling models.  Indeed, when $\Delta=0$, the free-fermion point, the six vertex model with domain wall boundary conditions is equivalent to uniformly random domino tilings of the Aztec diamond {\cite{EKLP92,FS06}}.
 Away from the free-fermion point, the situation is much less well understood unless at the so-called stochastic point where symmetric functions play an important role; see for example \cite{BCG16, OP17, Bor18, RS18, Agg20b, IMS20, Dim20b} for results on the stochastic six vertex model.
Nevertheless, for six vertex models with domain wall boundary conditions, a limit shape has been predicted using either the emptiness formulation probability  \cite{CP10, CPS16} or the tangent method \cite{CPS16b}.  
{From these predictions, it is expected that the six vertex model with domain wall boundary conditions should contain the equivalent of the frozen and rough regions when $|\Delta| < 1$. Simulations suggest that the six vertex model with domain wall boundary conditions contain three types of macroscopic regions when $\Delta <-1$, which is similar to the two-periodic Aztec diamond considered in \cite{CJ16, DK21}.}

Configurations of this model are in natural bijection with 
important combinatorial objects called \emph{alternating sign matrices} (ASMs).
ASMs are $n \times n$ matrices with entries in $\{0, \pm 1\}$ whose rows and columns sum to $1$ and whose nonzero entries alternate in sign along every row and column. From the point of view of the six-vertex model, studying ASMs corresponds to setting $a = b = c$ and hence $\Delta = 1/2$.
There has been some recent mathematical progress in the study of large random ASMs. 
Aggarwal, in \cite{Agg:20} has given a rigorous proof of the tangent method thus confirming Colomo and Pronko's prediction.  Gorin established the first fluctuation type result for ASMs by showing that the GUE corner process is the limit at the tangency points \cite{Gor14}, which has been generalized to other six vertex models \cite{Dim20}. The analogous result for Aztec diamonds was proved in \cite{JN06}.
Universality of the GUE corner process has very recently been established for uniformly random lozenge tilings at the tangency points \cite{AG21}.

In this paper, we consider a different type of limit for alternating sign matrices where the discrete process is no longer visible by moving away from the tangency points.   We introduce a directed path picture for the alternating sign matrices and show that the fluctuations of the maximum of the top path, which separates the ordered and disordered regions, converges to the GOE Tracy--Widom  distribution after suitable centering and rescaling.  This gives strong evidence that the top path should converge to the Airy-2 process after suitable centering and rescaling. The reason for this is the fact that the distribution of the maximum of an Airy process minus a parabola has the GOE Tracy--Widom distribution, see \cite{Jo03b,CQR11}.
To our knowledge this is the first edge fluctuation result away from the tangency points in a domain-wall six-vertex model when we are not in the free fermion case.

Alternating sign matrices arose naturally in the computation of the $\lambda$-determinant \cite{RR:86}, a generalization of Dodgson's condensation method for computing determinants.
Mills--Robbins--Rumsey~\cite{MRR83} conjectured that the number of ASMs of size $n$ is given by the explicit product formula,
\begin{equation}
\label{asm-formula}
\prod_{i=0}^{n-1} \frac{(3i+1)!}{(n+i)!}.
\end{equation}
Surprisingly, this formula was already known in the literature as the number of \tsscpp s (TSSCPPs) inside a $2n \times 2n \times 2n$ box, proved earlier by Andrews~\cite{Andrews-1994}.
This became known as the \emph{alternating sign matrix (ASM) conjecture} and was eventually settled first by Zeilberger~\cite{Zei96a} and then by Kuperberg~\cite{Kup96} using very different methods.
In the same paper, Mills--Robbins--Rumsey conjectured something more general.
They first realized both ASMs and TSSCPPs equivalently in terms of triangular arrays of size $n$, which Zeilberger called gog and magog triangles respectively. Then the ASM conjecture was equivalently a statement about these two families of triangular arrays being enumerated by \eqref{asm-formula}. See \eqref{eg:gogs2} and \eqref{eg-magogs2} for examples of gog and magog triangles of size $3$. Although the rules defining these are very similar, no bijective proof between these two objects has been found to this day. This is despite the proof that these families have the same enumeration under two refined statistics~\cite{fonseca-zinn-2008} and an explicit partial bijection between `large' subsets~\cite{AyyCorGou11}.

Mills--Robbins--Rumsey then generalized these triangular arrays to trapezoidal arrays for both families.
Zeilberger~\cite{Zei96a} proved the ASM conjecture by showing that these two families of trapezoidal arrays, called gog and magog trapezoids, are equinumerous, using the method of constant-term identities. A simplified version of his original proof is given in ~\cite{fischer-2016}.
We note that Krattenthaler defined more general families of gog and magog trapezoids and conjectured that they are equinumerous~\cite{krattenthaler-1996,krattenthaler-2016}. As of this writing, his conjecture is still open. Fischer~\cite{fischer-2018} has made some progress in proving Krattenthaler's conjecture.

Our results rely exclusively on Zeilberger's proof \cite{Zei96a} for the number of gog and magog trapezoids being equal in cardinality.  
Although Kuperberg's proof exploited the link between ASMs and the six-vertex model with domain wall boundary conditions, his proof does not generalize in any easy way to the result we need.
We then study the fluctuations at the `free' boundary of the \tsscpp s using formulas from \cite{AC:20}. These fluctuations turn out to have GOE Tracy--Widom distribution after suitable centering and rescaling.
{We remark that the idea of expressing an observable for a non-free-fermionic model in terms of some other observable of a related free-fermionic model has appeared before in other contexts~\cite{Bor18,IMS20}.}

In \cref{sec:Formulation}, we formulate our main result and provide an overview of the rest of the paper.  

\subsection*{Acknowledgements} 
{We thank the referees for their careful reading and useful comments.}
AA and SC acknowledge support from the Royal Society grant  IES\textbackslash R1\textbackslash 191139.
AA was partially supported by the UGC Centre for Advanced Studies and by Department of Science and Technology grant EMR/2016/006624. 
SC was supported by EPSRC EP\textbackslash T004290\textbackslash 1. 
KJ acknowledges support from the grant KAW
2015.0270 of the Knut and Alice Wallenberg Foundation.

\section{Formulations and the main result} \label{sec:Formulation}

To state our main result precisely, we first introduce {a directed path variant of alternating sign matrices which we call path corner sum matrices in \cref{subsec:directed}. We then define the limiting object, the GOE Tracy--Widom distribution in \cref{subsec:TWGOE}.  
Finally, we state our main result in \cref{subsec:main}.}  

\subsection{Directed Paths in the ASMs}
\label{subsec:directed}

\begin{defn}
An \emph{alternating sign matrix (ASM)} of {order} $n$ is an $n \times n$ matrix with entries in $\{0,1,-1\}$ such that
\begin{itemize}
\item the sum of the entries in each row and column equals $1$,
\item non-zero entries in each row and column alternate in sign.
\end{itemize}
\end{defn}

We denote $\mathcal{A}_n$ to be all the alternating sign matrices of order $n$. As an example, the seven alternating sign matrices of order $3$ are
\begin{equation}
\label{eg-asm3}
\begin{array}{ccccccc}
\hspace*{-0.3cm}\left( \begin{array}{rrr}
 1 &  0 & 0 \\
 0 & 1 &  0 \\
0 &  0 &  1 
\end{array} \right) &&
\hspace*{-0.9cm}\left( \begin{array}{rrr}
 1 &  0 & 0 \\
0 &  0 &  1 \\
 0 & 1 &  0 
\end{array} \right) &&
\hspace*{-0.9cm}\left( \begin{array}{rrr}
 0 & 1 &  0 \\
 1 &  0 & 0 \\
0 &  0 &  1 
\end{array} \right) &&
\hspace*{-0.9cm}\left( \begin{array}{rrr}
 0 & 1 &  0 \\
1 & -1 & 1 \\
 0 & 1 &  0 
\end{array} \right)
\\
&
\hspace*{-0.9cm}\left( \begin{array}{rrr}
 0 & 1 &  0 \\
0 &  0 &  1 \\
 1 &  0 & 0 
\end{array} \right) &&
\hspace*{-0.9cm}\left( \begin{array}{rrr}
0 &  0 &  1 \\
 1 &  0 & 0 \\
 0 & 1 &  0 
\end{array} \right) &&
\hspace*{-0.9cm}\left( \begin{array}{rrr}
0 &  0 &  1 \\
 0 & 1 &  0 \\
 1 &  0 & 0 
\end{array} \right).
\end{array}
\end{equation}

\begin{defn}
	A \emph{path corner sum matrix (PCSM)} of size $n$ is an $n \times n$ matrix $C=(c_{i,j})_{1 \leq i,j \leq n}$  with
\begin{itemize}
	\item $  c_{1,i} \in \{n-1, n\}$ and $ c_{i,n} \in \{n-1, n\}$ for $1 \leq i \leq n $
	\item { {$c_{i,j+1}-c_{i,j} \in \{0,1\}$ for $1\leq i\leq n$, $1 \leq j \leq n-1$, and  
 $c_{i,j}-c_{i+1,j} \in \{0,1\}$} for $1 \leq i\leq n-1$ and $1 \leq j \leq n$}.
\end{itemize}
\end{defn}
We denote the set of these matrices of size $n$ by $\mathcal{C}_{n}$ and call them \emph{path corner sum matrices}. These matrices are closely related to the corner sum matrices introduced in~\cite[Lemma 1]{RR:86}; see also \cite{propp-2001}. 
The shift of $n+1 \mapsto n$ is natural from our point of view, and it also fits in with our conventions for \tsscpp s introduced in the next section.

\begin{prop}
\label{prop:pcsm-bij}
PCSMs of size $n$ are in natural bijection with ASMs of size $n+1$
{via an affine transformation}.
\end{prop}

\begin{proof}
Given $A = (a_{i,j})_{1 \leq i,j \leq n+1} \in \mathcal{A}_{n+1}$, construct a matrix $C = (c_{i,j})_{1 \leq i,j \leq n}$ by
\begin{equation}
{{c}_{i,j}=n-\sum_{\substack{ 1 \leq r \leq i \\ 1 \leq s \leq n+1-j}} a_{r,s}}, \quad 1 \leq i,j \leq n.
\end{equation}
It is routine to check that $C \in \mathcal{C}_{n}$.  
It is also easy to verify that the inverse map is given by
\begin{equation}
a_{i,j} = c_{i-1,n+1-j} - c_{i,n+1-j} - c_{i-1,n+2-j} + c_{i,n+2-j},
\quad 1 \leq i,j \leq n+1,
\end{equation}
where we assume $c_{i,n+1} {= c_{0,i}} = n$ 
{and $c_{i,0} = c_{n+1,n+1-i} = n-i$
for $0 \leq i \leq n+1$; see \cref{fig:G_3}}.
\end{proof}

The seven path corner sum matrices of size $2$, $\mathcal{C}_2$, are given in the same order as the corresponding ASMs in \eqref{eg-asm3} by
\begin{equation}
\label{eg-pcsm3}
\begin{array}{ccccccc}
\left( \begin{array}{rrr}
 1 &  1  \\
 0 & 1 
\end{array} \right) &&
\hspace*{-0.9cm}\left( \begin{array}{rrr}
 1 &  1  \\
1 &  1 
\end{array} \right) &&
\hspace*{-0.9cm}\left( \begin{array}{rrr}
 1 & 2 \\
 0 &  1
\end{array} \right) &&
\hspace*{-0.9cm}\left( \begin{array}{rrr}
 1 & 2  \\
1 & 1 
\end{array} \right)
\\
&
\hspace*{-0.9cm}\left( \begin{array}{rrr}
1 &  2 \\
 1 &  2
\end{array} \right) &&
\hspace*{-0.9cm}\left( \begin{array}{rrr}
 2 & 2 \\
1 &  1 
\end{array} \right) &&
\hspace*{-0.9cm}\left( \begin{array}{rrr}
2 &  2  \\
 1 & 2 
\end{array} \right).
\end{array}
\end{equation}

We think of the matrix entries of the PCSMs as heights.  
{For each PCSM, we consider a family of directed paths starting from the leftmost column and ending on the bottom row using north-east and south-east paths. The paths separate entries which differ by one. See \cref{fig:size10path} for an example. 
 In our description, the top-rightmost path in this family  gives a full description of the boundary of the disordered region in the top-left quadrant of ASMs.  We are interested in precisely this path. This path is not easily seen from an ASM configuration and it can be thought of as an analogue of the DR-paths used to investigate the asymptotics of uniformly random domino tilings of the Aztec diamond~\cite{Jo03}.
}

For ease of notation, let $[i]_2=i \mod 2$ for an integer $i$.
To describe these directed paths precisely, we introduce the graph, $G_n^\mathbf{g}$, whose vertex set is given by
\begin{equation}
V_n^{\mathbf{g}} = V_n^{\mathbf{g},1} \cup V_n^{\mathbf{g},2},
\end{equation}
where 
\begin{equation}
V_n^{\mathbf{g},1}= \left\{(i,j) \left| \substack{\ds [i]_2=1,[j]_2=0, -1 \leq i \leq 2n-1, \\ \ds 0 \leq j \leq 2n, (i,j) \not =(-1,0)}  \right. \right\},
\end{equation}
and
\begin{equation}
V_n^{\mathbf{g},2} = \left\{(i,j) \left| \substack{\ds [i]_2=0,[j]_2=1, 0 \leq i \leq 2n, \\ \ds 1 \leq j \leq 2n+1, (i,j) \not =(2n,2n+1) } \right. \right\}.
\end{equation}
The edge set of $G_n^{\mathbf{g}}$ is described as follows:
\begin{itemize}
\item for each $v \in V_n^{\mathbf{g},1}\backslash \{(2n-1,2n)\}$, there is an edge $(v,v+(1,1))$, 
\item for each $v \in V_n^{\mathbf{g},1}\backslash \{(-1,2j):1\leq j \leq n\}$ there is an edge $(v,v+(-1,1))$,
\item for each $v \in V_n^{\mathbf{g},1}\backslash \{(2i-1,0):1\leq i \leq n \}$ there is an edge $(v,v+(1,-1))$.
\end{itemize}
\cref{fig:G_3}(a) shows $G_3^{\mathbf{g}}$.

\begin{center} 
\begin{figure}
\begin{tabular}{c c}
\includegraphics[height=3cm]{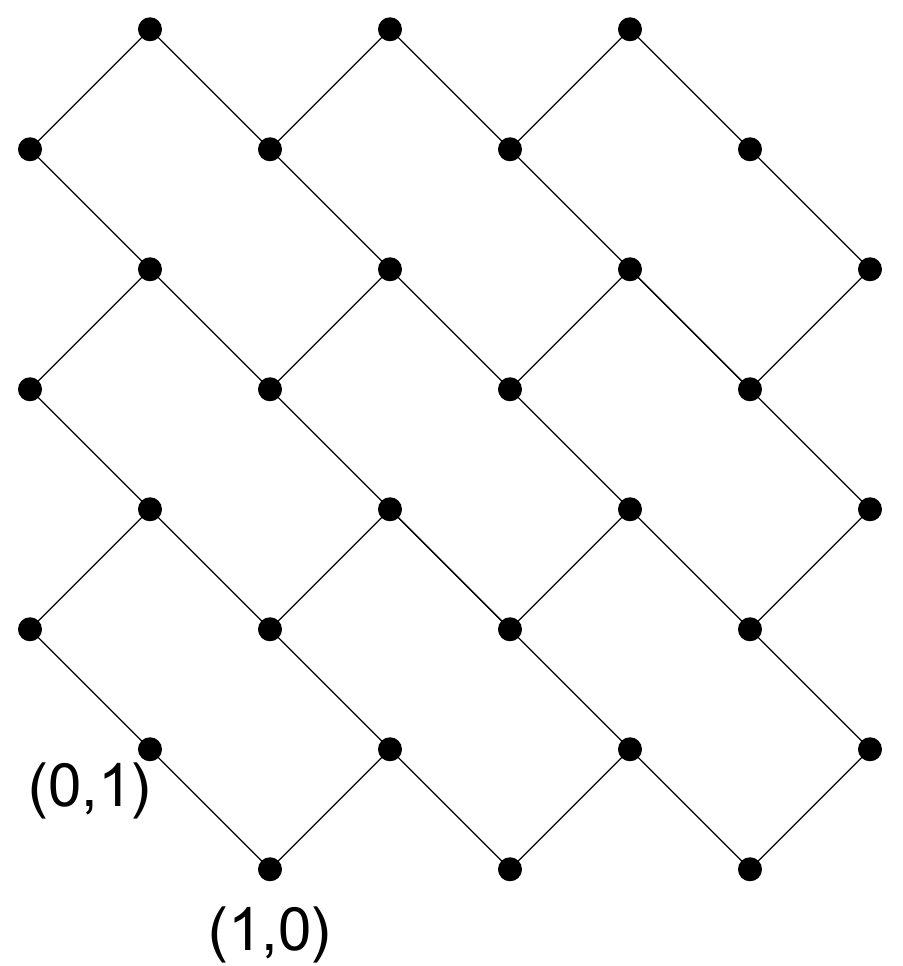}
&
\hspace*{1cm} \includegraphics[height=3cm]{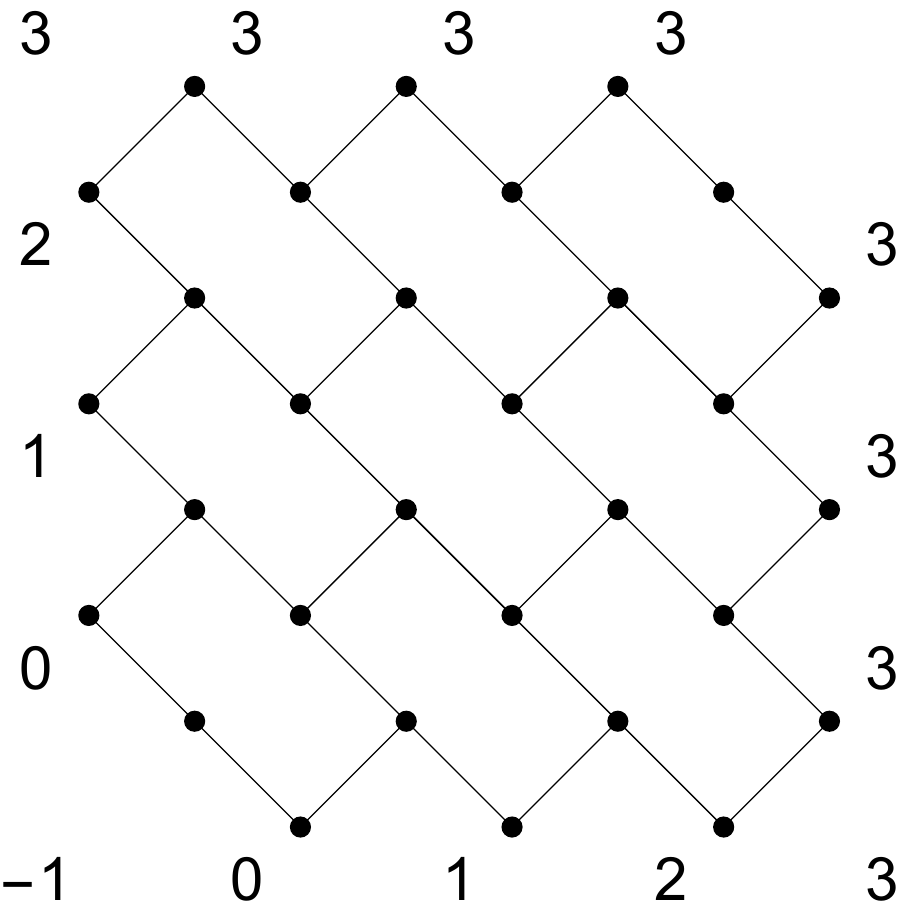} \\
{(a)} & (b)
\end{tabular}
\caption{The left figure shows $G_3^{\mathbf{g}}$ the Cartesian coordinates. The right figure shows the external heights on external faces sharing an edge with $G_3^{\mathbf{g}}$.}
\label{fig:G_3}
\end{figure}
\end{center}

 The directed paths travel on the edges of this graph in the following way:   paths can travel in the direction $(1,-1)$ on edges parallel to $(1,-1)$ and can travel in either direction on edges parallel to $(1,1)$.  The paths start at the vertices $\{(-1,{2i}): 1 \leq i \leq n \}$ and terminate at the vertices $\{(2i-1,0):1 \leq i \leq n \}$. Paths cannot meet at vertices.  
We superimpose the entries of $\mathcal{C}_n$ onto the faces of $G_n^{\mathbf{g}}$ so that for $({c}_{i,j})_{1 \leq i,j \leq n} \in \mathcal{C}_n$, ${c}_{i,j}$ is assigned to the face whose center is $(  2j-\frac{3}{2},2n-2i+\frac{3}{2})$.  These superimposed numbers represent the height of that face. The heights of the faces external to $G_n^{\mathbf{g}}$ but sharing an edge {with it} are given by
\begin{itemize}
\item height $i-1$ for the faces {centered at} $ (2i-\frac{3}{2},-\frac{1}{2})$ with $1\leq i \leq n+1$,
\item height ${j-1}$ for the faces {centered at} $ (-\frac{3}{2},2j-\frac{1}{2})$ with $1 \leq j \leq n+1$, and
\item height $n$ otherwise.
\end{itemize}
{The top directed path is the path starting from $(-1,2n)$ and ending at $(2n-1,0)$. }

\cref{fig:G_3}(b) shows heights on the external bordering faces to $G_3^{\mathbf{g}}$.  
The level lines of the height function on $G_n^{\mathbf{g}}$, i.e. the height differences between faces, represent the directed paths on $G_n^{\mathbf{g}}$.  The directed paths for $\mathcal{C}_2$ are shown in \cref{fig:size2paths} while the directed paths for a PCSM configuration of size 9 is given in \cref{fig:size10path}. 

\begin{center}
\begin{figure}[htbp!]
\includegraphics[height=6cm]{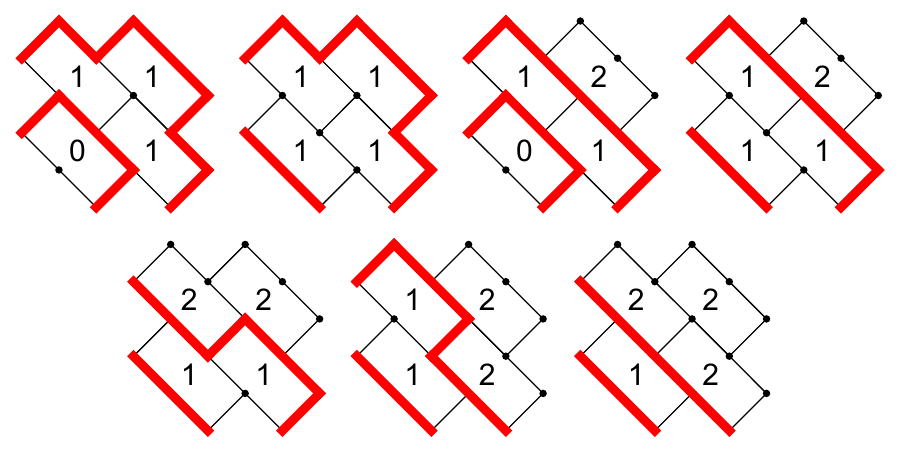}
\caption{The directed paths for $\mathcal{C}_2$ with the same ordering as used for depicting $\mathcal{A}_3$, $\mathcal{C}_3$ and $\mathcal{C}_2$. Paths are oriented from the left side.  }
\label{fig:size2paths}
\end{figure}
\end{center}

{
We define an alternate coordinate system for vertices in $V_n^{\mathbf{g}}$. If $(x,y)$ is a vertex, let $t=(x-y+1)/2$ and $h=((y-2n)+(x+1))/2-1$. Thus, every vertex can be expressed using $(t,h)$ coordinates.
We will be interested in the {vertices of the} top {directed} path which starts at $(-1,2n)$ and ends at $(2n-1,0)$. Note that the vertices $(x,y)$ and $(x-1,y-1)$ have the same $t$-coordinate, so if the directed path passes through both of these vertices, we omit the $t$-coordinate for the $(x-1,y-1)$ vertex and retain the $t$-coordinate for $(x,y)$.  Then, reading from top-left to bottom right, the $t$-coordinates for such vertices are $(-n,\dots,-1,0,1,\dots,n)$. We then define the random vector $T_n$ as
the $h$-coordinates for the top directed path indexed by the $t$-coordinates,
\begin{equation}\label{eq:TnAiry}
T_n = (T_n(-n),\dots,T_n(-1),T_n(0),T_n(1),\dots,T_n(n)).
\end{equation}
{For the example in \cref{fig:size10path},
\[
T_9 = (0,1,2,3,4,4,4,4,4,5,5,5,4,4,4,3,2,1,0).
\]
 }
Our main result, \cref{thm:mainthm}, is a statement about $\max (T_n)$ and our conjecture, \cref{conj:GUE}, is a statement about $T_n(0)$ {as well as the rest of the path}.
}

\begin{center}
\begin{figure}[htbp!]
\begin{tabular}{c c}
\includegraphics[height=5cm]{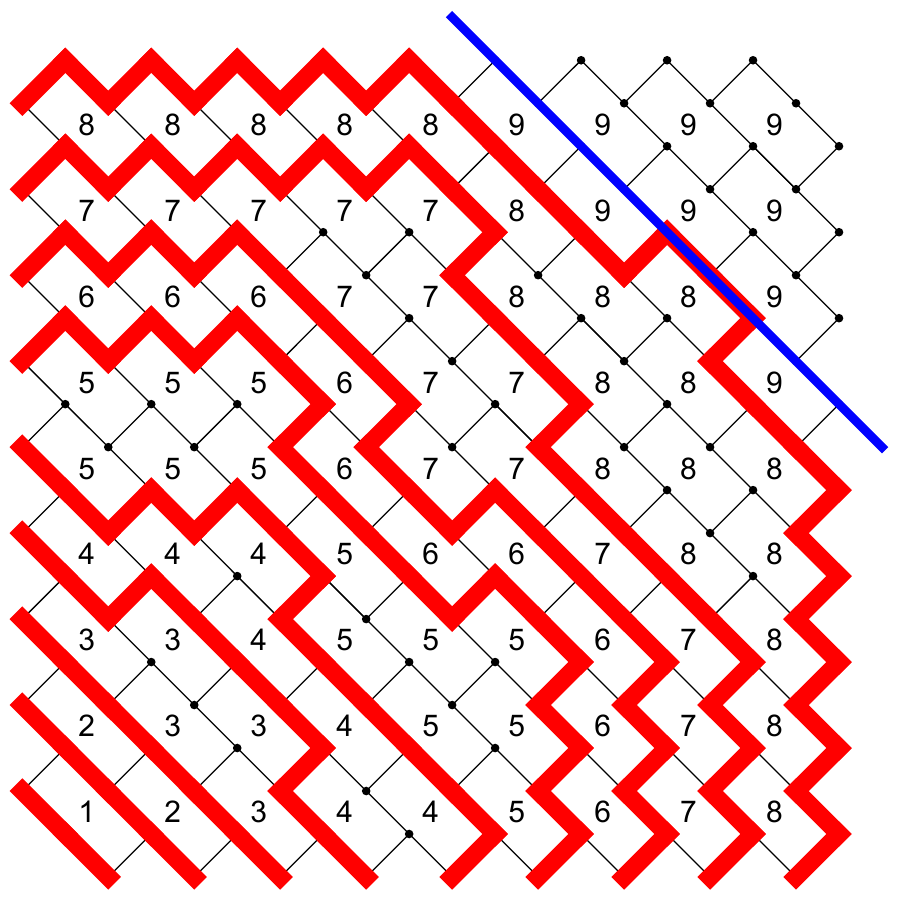} 
    &
\includegraphics[height=4.5cm]{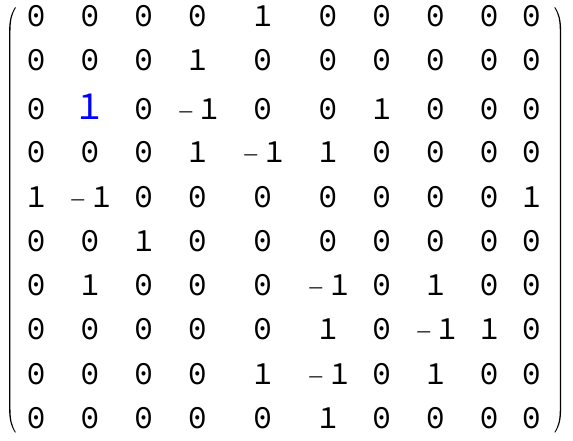} \\
(a) & (b)
\end{tabular}
\caption{{(a) A path corner sum matrix of size $9$ along with its directed path representation.  Paths are oriented from the left side. Here, we have $\max T_9= 5$ and is overlaid in blue and bold. (b) The corresponding ASM is on the right. The $1$ closest to the top-left corner is marked in blue and corresponds to the intersection of the blue diagonal and the top path in (a).}}
\label{fig:size10path}
\end{figure}
\end{center}

\subsection{GOE Tracy--Widom distribution}
\label{subsec:TWGOE}

We define the GOE Tracy--Wid\-om distribution~\cite{TW96} 
through a Fredholm Pfaffian~\cite{Ra00,Fer04,OQR:17,BBCS:18,BBNV18}.

The Pfaffian of an {even} anti-symmetric matrix $A=(a_{i,j})_{1 \leq i,j \leq 2k}$ is given by
\begin{equation}
\mathrm{Pf}(A)= \frac{1}{2^k k!} \sum_{\sigma \in \mathcal{S}_{2k}} \mathrm{sgn}(\sigma) a_{\sigma(1),\sigma(2)} \dots a_{\sigma(2k-1),\sigma(2k)},
\end{equation}
where $\mathcal{S}_{2k}$ is the set of permutations of $\{1,\dots, 2k\}$.  Let $\mathrm{Ai}(x)$ denote the Airy function, that is,
\begin{equation}
\mathrm{Ai}(x)=\frac{1}{2 \pi \mathrm{i}} \int_{\Gamma_{\mathrm{Ai}}} \text{d}z \; e^{\frac{z^3}{3}-x z},
\end{equation}
where $\Gamma_{\mathrm{Ai}}$ is the union of two semi-infinite rays starting from the origin in directions $e^{\frac{\mathrm{i}\pi}{3}}$ and $e^{-\frac{\mathrm{i}\pi}{3}}$, with the contour being oriented from $\infty e^{-\frac{\mathrm{i}\pi}{3}}$ to $\infty e^{\frac{\mathrm{i}\pi}{3}}$ with $\mathrm{i}=\sqrt{-1}$. 

Introduce the following 2 by 2 block kernel
\begin{equation}
\mathbf{K}_{\mathrm{GOE}}(x,y) = \left( \begin{array}{cc}
	K_{\mathrm{GOE}}^{11}(x,y) & K_{\mathrm{GOE}}^{12}(x,y) \\
	K_{\mathrm{GOE}}^{21}(x,y) & K_{\mathrm{GOE}}^{22}(x,y) \end{array}\right)
\end{equation}
where
\begin{equation}\label{eq:KGOE11}
K_{\mathrm{GOE}}^{11}(x,y)= \frac{1}{4}\int_0^\infty \text{d} \lambda \;
\left( \mathrm{Ai}(x+\lambda)\mathrm{Ai}'(y+\lambda)- \mathrm{Ai}'(x+\lambda)\mathrm{Ai}(y+\lambda) \right),
\end{equation}
\begin{equation}\label{eq:KGOE12}
K_{\mathrm{GOE}}^{12}(x,y)= \int_0^\infty \text{d} \lambda \; \mathrm{Ai}(x+\lambda)\mathrm{Ai}(y+\lambda)+\frac{1}{2} \mathrm{Ai}(x) \int_0^\infty \text{d} \lambda \; \mathrm{Ai}(y-\lambda) 
\end{equation}
\begin{equation}
K_{\mathrm{GOE}}^{21}(x,y)=-K_{\mathrm{GOE}}^{12}(y,x),
\end{equation}
\begin{equation}
\begin{split}\label{eq:KGOE22}
K_{\mathrm{GOE}}^{22}(x,y) =& \int_0^\infty \text{d} \lambda \int_\lambda^\infty \text{d} \mu \; \mathrm{Ai}(x+\lambda)\mathrm{Ai}(y+\mu)-\mathrm{Ai}(x+\mu)\mathrm{Ai}(y+\lambda)
\\&-\int_0^\infty \text{d} \mu \; \mathrm{Ai}(x+\mu)+ \int_0^\infty \text{d} \mu \; \mathrm{Ai}(y+\mu)-\mathrm{sgn}(x-y),
\end{split}
\end{equation}
where $\mathrm{sgn}$ {is the standard signum function}.

The GOE Tracy--Widom distribution is defined through a Fredholm Pfaffian by
\begin{equation}
\label{eq:GOETW}
\begin{split}
F_1(s) =& \mathrm{Pf}(\mathbbm{J}-\mathbf{K}_{\mathrm{GOE}})_{L^2(s,\infty)} \\ 
=& 1+\sum_{k=1}^{\infty} \frac{(-1)^k}{k!} \int_s^{\infty} \text{d}x_1 \dots \int_s^\infty \text{d}x_k \; \mathrm{Pf} (\mathbf{K}_{\mathrm{GOE}}(x_i,x_j))_{1 \leq i,j\leq k},
\end{split}
\end{equation}
where
\begin{equation}
\mathbbm{J}(x,y) = \left( \begin{array}{cc}
0 & 1 \\
-1 & 0 \end{array} \right) \mathbbm{I}_{x=y}.
\end{equation}
We mention that $F_1(s)$ is well-defined since the series converges and that the GOE Tracy--Widom distribution correlation kernel is not unique. Ours differs from the usual one in having a factor $1/4$ in $K_{\mathrm{GOE}}^{11}$ instead of in $K_{\mathrm{GOE}}^{22}$; our formulation is similar to the
ones used in~\cite[Equations (2.9) and (6.17)]{Fer04} and~\cite{BBCS:18}.   Since the term $K_{\mathrm{GOE}}^{11}$ is always paired to a $K_{\mathrm{GOE}}^{22}$ in the expansion of the Pfaffian, our representation above is equivalent to the usual one.

\subsection{Statement of the main result}
\label{subsec:main}

We are now ready to give our main result.  Let $Z_n$ denote $|\mathcal{C}_n|$, that is the number of path corner sum matrices of size $n$. Introduce the following probability measure\footnote{the superscript $\mathbf{g}$ refers to gogs defined later; see \cref{rem:notation} for an explanation} for ${C}\in \mathcal{C}_n$
$$
\mathbb{P}^{\mathbf{g}}_n[{C} ]= \frac{1}{Z_n}.
$$
We only consider the behavior in the top right quadrant of $G_n^{\mathbf{g}}$, that is the behavior in the region {$[n,2n] \times [n,2n]$}. The behavior in the other four quadrants is similar.  
Above the top path in $G_n^{\mathbf{g}}$ described above, the configuration is frozen.  These correspond to entries equal to $n$ in PCSMs and a connected region of zeros in the ASMs which is also connected to the boundary (that is the first row and column of the ASMs).  
 Rescaling the  $G_n^{\mathbf{g}}$ by $n \times n$ so that it fits into {$[0,2] \times [0,2]$}, the equation of the top path converges to 
$$
x(2-x)+y(2-y)+(2-x)(2-y)=1
$$
for $x,y \in[1,2]$.   As mentioned in \cref{sec:Introduction}, this limiting curve, {under an affine transformation}, was
 predicted in~\cite{CP10, CPS16b} and proved in \cite{Agg:20}.

Introduce the constants 
\begin{equation} \label{eq:scalings}
\alpha=2-\sqrt{3}, \hspace{5mm} c_1=\frac{2}{3^{4/3}},  \mbox{ and }c_0=\frac{1}{3\sqrt{3}c_1}=\frac{1}{2 \cdot 3^{1/6}},
\end{equation}
which will be used throughout the paper. Our main theorem is the following.

\begin{thm}
\label{thm:mainthm}
We have that
\begin{equation}
\lim_{n\to \infty} \mathbb{P}^{\mathbf{g}}_n\left[  \frac{\max (T_n)-(1-\alpha)n}{c_0 n^{\frac{1}{3}}}\leq s\right]= F_1(s),
\end{equation}
where $F_1$ is defined in~\eqref{eq:GOETW}. 
\end{thm}

In words, the random variable $\max T_n$ represents  the frozen triangle in the top right quadrant of the PCSMs, where the entries are equal to $n$ in this frozen triangle. It also represents the maximum of the top path of $G_n^{\mathbf{g}}$.  Our result says that the fluctuations of the maximum of the top path converges weakly to the GOE Tracy--Widom distribution after centering and rescaling.  In terms of alternating sign matrices, our result says that the fluctuations of the frozen triangle (that is all entries equal to 0 in this triangle)  in one of the four quadrants of the ASMs converges weakly to the GOE Tracy--Widom distribution after centering and rescaling.  By symmetry, this occurs for all four quadrants of the ASMs.  \cref{fig:simulation} shows a simulation. { Note that our result only recovers the arctic curve at four points and does not recover the full curve. }

\begin{figure}
\includegraphics[height=9cm]{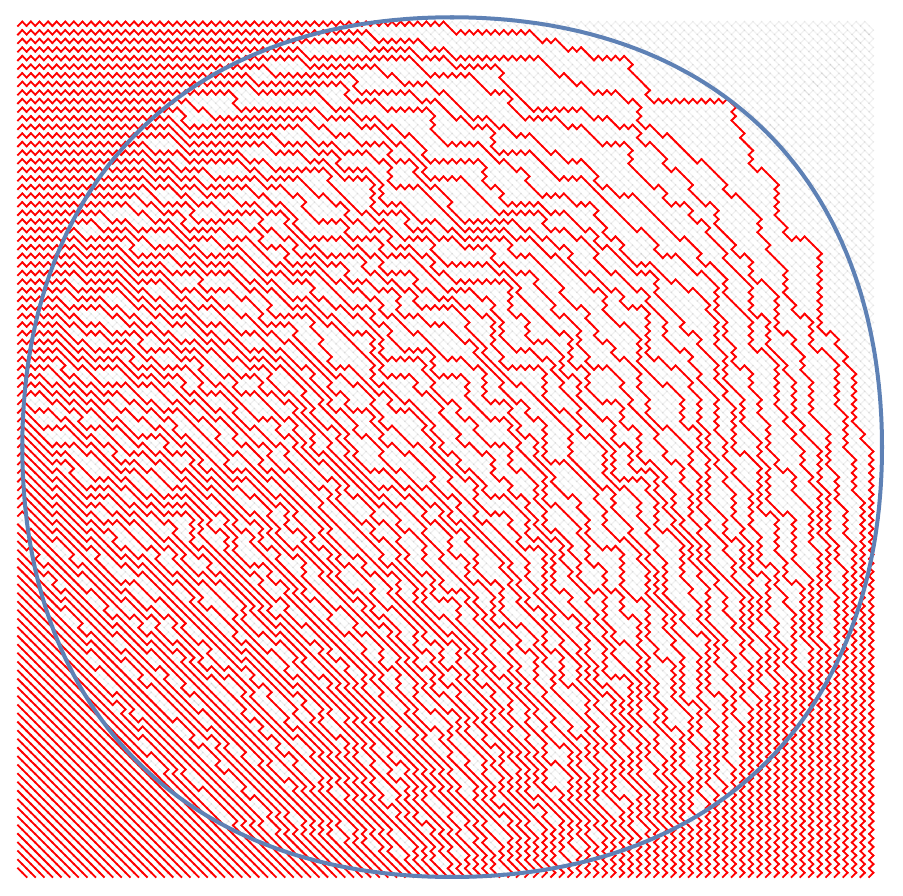}
\caption{A simulation of the directed paths for $n=100$ made from Glauber dynamics. The limit shape is overlaid.   } 
\label{fig:simulation}
\end{figure}

It is known that the distribution of the maximum of the Airy-2 process minus by a parabola is given by  GOE Tracy--Widom distribution \cite{Jo03b, CQR11}.  As we have established that the maximum of the top path has GOE Tracy--Widom fluctuations, we expect that the original path should be given by the Airy process after suitable centering and rescaling.  We expect that the maximum should be obtained close to $t=0$ and the scale factor to be the same at this point. This leads to the following conjecture.

\begin{conj}
\label{conj:GUE}
    After appropriate rescaling, $T_n$, defined in~\eqref{eq:TnAiry}, converges to the Airy-2-process, and in particular,
\begin{equation}
    \lim_{n\to \infty} \mathbb{P}^{\mathbf{g}}_n\left[  \frac{ T_n(0)-(1-\alpha)n}{4^{\frac{1}{3}}c_0 n^{\frac{1}{3}}}\leq s\right]= F_2(s),
\end{equation}
where $F_2$ is the GUE Tracy--Widom distribution \cite{TW94}. 
\end{conj}

Note that $T_n(0)$ is very closely related to the emptiness formation probability; see~\cite{CP10,CPS16} for details on the emptiness formation probability.

To prove \cref{thm:mainthm}, we first compare $\max (T_n)$ with a specific {random variable}, 
which will be denoted by  $X_n^{\mathbf{m}}$, in the \tsscpp {} of size $n$, also defined below.  In fact, we will show that these two objects are equal in distribution, using Zeilberger's deep proof of the ASM conjecture~\cite{Zei96a}. 
{The strategy of the proof is as follows. We will prove \cref{thm:mainthm} in \cref{sec:Comparison} assuming \cref{thm:thmtsscpp} about an asymptotic result on a point process defined through \tsscpp s. 
In \cref{sec:asymptotics}, we introduce formulas for this point process in \cref{prop:discrete}  and give the proof of \cref{thm:thmtsscpp}.   
We give the proof of \cref{prop:discrete} in \cref{sec:derivation} using formulas from \cite{AC:20}.
}

\section{{ASMs and TSSCPPs}} 
\label{sec:Comparison}

\subsection{TSSCPPs} 
\label{subsec:TSSCPPs}
A \tsscpp{} (TSSCPP) of order $n$ {is a} rhombus tiling of a regular hexagon with side length $2n$ with the maximum possible symmetry. {All} the information about the tiling is contained in $(1/12)$'th of the hexagon~\cite[Section 8]{MRR86}. 
{A TSSCPP can be equivalently represented as a perfect matching of a certain graph, which we describe now, following the same conventions as in \cite{AC:20}.}

Define the graph $G_n^{\mathbf{m}} = (\mathtt{V}_n^{\mathbf{m}}, \mathtt{E}_n^{\mathbf{m}})$, for which the vertices are given by 
\begin{equation}	
	\mathtt{V}_n^{\mathbf{m}} = \{ (x_1,x_2) \mid 0 \leq x_1\leq 2n, x_1 \leq x_2 \leq 2n+1\} \backslash \{ (2n,2n+1) \mathbbm{1}_{n \in 2\mathbb{Z}+1} \} 
\end{equation}
and  the edges are given by 
\begin{equation}
	\begin{split}
		\mathtt{E}_n^{\mathbf{m}} =& \left\{ (( x_1,x_2),(x_1+1,x_2) \left|
		\substack{\ds 0\leq x_1 \leq 2n-1,x_1 \leq x_2 \leq 2n+1, \\ \ds [x_1+x_2]_2=1} \right. \right\} \\
 & \cup \{ (( x_1,x_2),(x_1,x_2+1)\mid 0\leq x_1 \leq 2n-\mathbbm{1}_{n \in 2\mathbb{Z}+1},x_1 \leq x_2 \leq 2n\} \\
		&\cup \{ ((x_1,x_1),(x_1+1,x_1+1))\mid 0\leq x_1 \leq 2n-1 \} ;
	\end{split}
\end{equation}
see \cref{fig:rectanglecoords}. We let $\mathbf{b}$ denote the vertex $(2n,2n+1-[n]_2)$.

\begin{figure}[htbp!]
\begin{center}
\begin{tabular}{cc}
	\includegraphics[height=3.5cm]{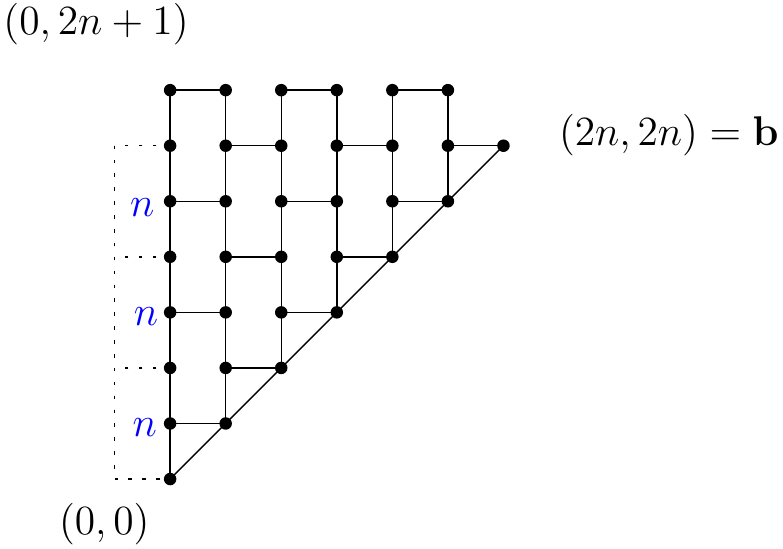} &
	\includegraphics[height=3.5cm]{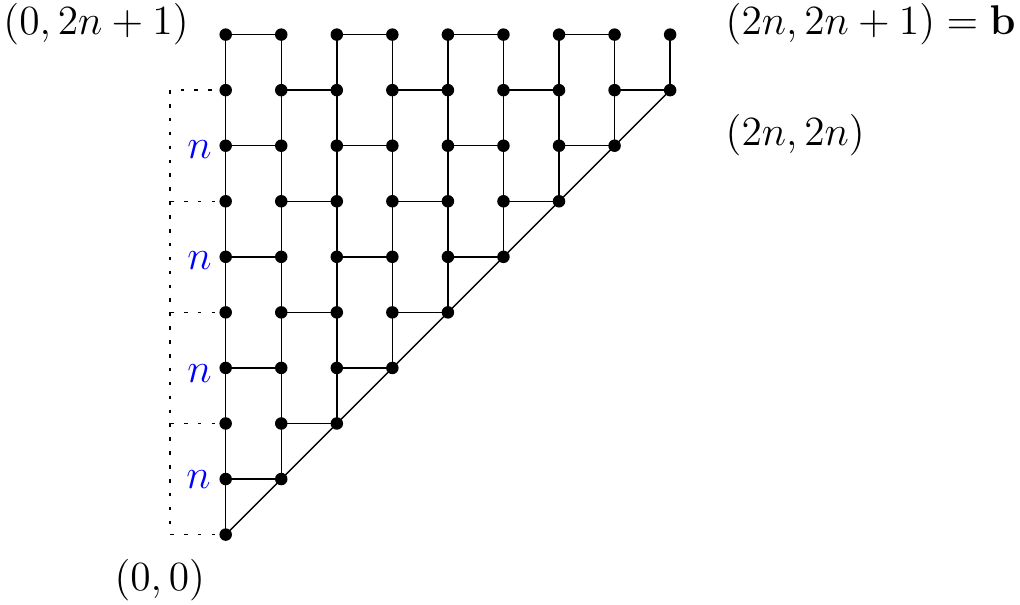} \\
	(a) & (b)
\end{tabular}
	\caption{The left figure shows $G^{\mathbf{m}}_3$, while the right figures shows  $G^{\mathbf{m}}_4$. The coordinates are for the top left, rightmost and bottom left vertices. For both, the height function on rectangular faces left of the graph is shown in blue. } 
	\label{fig:rectanglecoords}
\end{center}
\end{figure}
A dimer is an edge and a dimer covering is a subset of edges so that each vertex is incident to exactly one edge. 

\begin{prop}[{\cite[Proposition 2.2]{AC:20}}]
\label{prop:tsscpp-bij}
Dimer configurations of $G_n^{\mathbf{m}}$ are in bijection with TSSCPP configurations of size $n+1$
\end{prop}

There is a height function associated to TSSCPPs, even though the graph $G_n^{\mathbf{m}}$ is not bipartite.  This height function is defined on {rectangular} faces of the graph as follows: the height on the faces to the left of the TSSCPP, i.e. those on faces whose center is given by $(-\frac{1}{2},2j+1)$ for $0\leq j \leq n$ are given by $n$; see \cref{fig:rectanglecoords}.  Suppose the height at the face with center $(i+\frac{1}{2},j)$ is given by $t$. Then the height at the face whose center is $(i+\frac{3}{2},j+1)$ is $t$ if there is a dimer covering the edge $((i+1,j),(i+1,j+1))$ and is $t-1$ if there is no dimer covering the edge $((i+1,j),(i+1,j+1))$. The prescription of the faces to the left of the TSSCPP and the above rule define the heights on all the hexagonal faces of $G_n^{\mathbf{m}}$ for each dimer covering and ensures that the heights are integer-valued in $\{0,1,\dots,n\}$.  
See \cref{fig:TSSCPPsize2} for height functions of all TSSCPPs of size 2. 

\begin{figure}[htbp!]
\includegraphics[height=6cm]{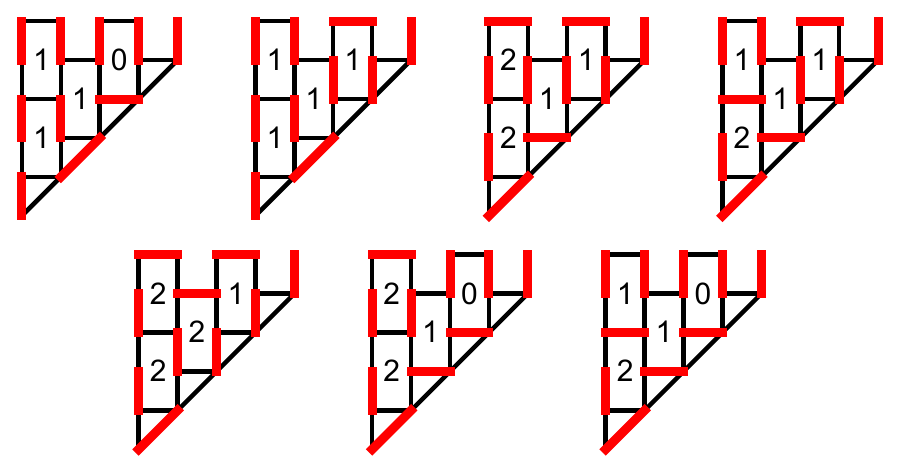}
\caption{The 7 TSSCPPs of size 2 along with their height functions.}
\label{fig:TSSCPPsize2}
\end{figure}

Each dimer configuration is chosen uniformly from the set of dimer configurations $\mathcal{M}(G_n^{\mathbf{m}})$, that is, for $M \in \mathcal{M}(G_n^{\mathbf{m}})$,
$$
\mathbb{P}^{\mathbf{m}}_n[M] = \frac{1}{Z_n},
$$
using the alternating sign matrix theorem~\cite{Zei96a,Kup96}.

\begin{rem} \label{rem:notation}
Since we focus on two different objects and their variations, we use the notation that $\mathbf{g}$ refers to any objects related to alternating sign matrices and $\mathbf{m}$ refers to any objects related to TSSCPPs.  These stand for gogs and magogs respectively which are introduced in the next subsection. 
\end{rem}

\subsection{Gogs and Magog trapezoids} 
\label{subsec:gogsmagogs}

\begin{defn}
\label{def:gog-triangle}
A \emph{monotone triangle} or \emph{gog triangle} of order $n$ is a triangular array $(g_{i,j})_{1 \leq j \leq i \leq n}$ of positive integers such that

\begin{itemize}
\item $g_{i,j} \leq g_{i-1,j} \leq g_{i,j+1}$ whenever all the entries are defined,

\item $g_{i,j} < g_{i,j+1}$ whenever both entries are defined,

\item $1 \leq g_{i,j} \leq n+1$, for $1 \leq j \leq i \leq n$.
\end{itemize}
\end{defn}

It is easy to see that monotone triangles of order $n$ are in natural bijection with $\mathcal{A}_{n+1}$.
For a matrix $A=(a_{i,j})_{1 \leq i,j \leq n+1} \in \mathcal{A}_{n+1}$, consider the 
$n \times (n+1)$ matrix $b$ defined by
\begin{equation}
b_{i,j} = {\sum_{1 \leq i' \leq i} a_{i',j}}, \quad i \in [n], j \in [n+1].
\end{equation}
One can see that $b$ is a $\{0,1\}$-matrix and in the $i$'th row, $b$ has exactly $i$ $1$'s. Arranging the column indices $j$ for which $b_{i,j} = 1$ in increasing order for each $i$ gives 
a triangular array which satisfies all the conditions of a gog triangle. The process is clearly bijective.
Our convention  to remove the bottom row of the triangular array differs from the literature~\cite{Bre99}, but makes it easier for our upcoming analysis. 
In particular, our monotone triangles of order $n$ are in natural bijection with PCSMs of size $n$.
In our running example, the monotone triangles in bijection with $\mathcal{A}_3$ are given in the same order as \eqref{eg-asm3} by
\begin{equation}
\label{eg:gogs2}
\begin{array}{ccccccc}
\hspace*{-0.3cm}\begin{array}{ccccc}
&&1&& \\
& 1 && 2 & \\
\end{array} &&
\hspace*{-0.9cm}\begin{array}{ccccc}
&&1&& \\
& 1 && 3 & \\
\end{array} &&
\hspace*{-0.9cm} \begin{array}{ccccc}
&&2&& \\
& 1 && 2 & \\
\end{array} &&
\hspace*{-0.9cm} \begin{array}{ccccc}
&&2&& \\
& 1 && 3 & \\
\end{array}
\\
&
\hspace*{-0.9cm}\begin{array}{ccccc}
&&2&& \\
& 2 && 3 & \\
\end{array} &&
\hspace*{-0.9cm} \begin{array}{ccccc}
&&3&& \\
& 1 && 3 & \\
\end{array} &&
\hspace*{-0.9cm} \begin{array}{ccccc}
&&3&& \\
& 2 && 3 & \\
\end{array}.
\end{array}
\end{equation}

\begin{defn}
An \emph{$(n,k)$-gog trapezoid} is a trapezoidal array 
\[
(g_{i,j})_{1 \leq i \leq n, 1 \leq j \leq \min(i,k)}
\]
of positive integers such that 

\begin{itemize}
\item $g_{i,j} \leq g_{i-1,j} \leq g_{i,j+1}$ whenever all the entries are defined,

\item $g_{i,j} < g_{i,j+1}$ whenever both entries are defined,

\item $g_{i,j} \leq n+1$ for all valid $i,j$.
\end{itemize}
\end{defn}

In other words, an $(n,k)$-gog trapezoid is a gog triangle with $n$ rows after removing an equilateral triangle of size {$n-k$} from the bottom right-most side of the gog triangle. 
We let $\mathtt{Gog}(n,k)$ to be the set of $(n,k)$-gog trapezoids.
An example of a configuration in  $\mathtt{Gog}(6,3)$ is 
$$
\begin{array}{ccccccccccccc}
&&&&& 6 \\
&&&&5 && 6  \\
&&&3 && 5 && 7 \\
&& 3 &&4 &&6 &&  \\
&2 && 3&&5 && \\
2&&3&&4&& \\
\end{array}.
$$

We now describe triangular arrays in bijection with TSSCPPs.

\begin{defn}
A \emph{magog triangle} of order $n$ is a triangular array 
\[
(m_{i,j})_{1 \leq j \leq i \leq n}
\]
of positive integers such that

\begin{itemize}
\item $m_{i,j} \leq m_{i+1,j+1}$ whenever both entries are defined,

\item $m_{i,j} \geq m_{i+1,j}$ whenever both entries are defined,

\item $m_{i,j} \leq j+1$ for all valid $i,j$.
\end{itemize}
\end{defn}

Again, our convention differs from the literature in that we drop the leftmost column from the standard definition in the literature; e.g. see~\cite{Bre99}. 
Magog triangles of size 2 are given by 
\begin{equation}
\label{eg-magogs2}
\begin{array}{ccccccc}
\hspace*{-0.4cm}\begin{array}{ccccc}
&&2&& \\
& 2 && 3 & \\
\end{array} &&
\hspace*{-0.9cm} \begin{array}{ccccc}
&&2&& \\
& 2 && 2 & \\
\end{array} &&
\hspace*{-0.9cm} \begin{array}{ccccc}
&&1&& \\
& 1 && 2 & \\
\end{array} &&
\hspace*{-0.9cm} \begin{array}{ccccc}
&&2&& \\
& 1 && 2 & \\
\end{array}
\\
&
\hspace*{-0.9cm} \begin{array}{ccccc}
&&1&& \\
& 1 && 1 & \\
\end{array} &&
\hspace*{-0.9cm} \begin{array}{ccccc}
&&1&& \\
& 1 && 3 & \\
\end{array} &&
\hspace*{-0.9cm} \begin{array}{ccccc}
&&2&& \\
& 1 && 3 & \\
\end{array}
\end{array}
\end{equation}

\begin{prop}
\label{prop:magog-bij}
Magog triangles of order $n$ are in natural bijection with TSSCPPs of size $n$.
\end{prop}

\begin{proof}
By \cref{prop:tsscpp-bij}, it suffices to consider perfect matchings of $G^{\mathbf{m}}_n$.
For such a perfect matching, build the $k$'th row, $1 \leq k \leq n$, of the magog triangle by considering vertical edges of the form $\{(i,i+2(n-k)+1),(i,i+2(n-k)+2): 0 \leq i \leq 2k-1 \}$ and listing the $x$-coordinates of the matched edges in increasing order. 
From this triangular array, subtract entrywise the triangular array
$$
\begin{array}{ccccccccccccc}
&&&&& -1 \\
&&&&-1 && 0  \\
&&&-1 && 0 && 1 \\
&& \iddots &&\iddots &&\ddots &&\ddots &&  \\
&-1 && 0 && \cdots  &&  n-3 && n-2.
\end{array}
$$
{
The interlacing property of the vertical edges of the perfect matching ensures that this operation gives a magog triangle. The fact that the vertical edges determine the perfect matching ensures that this operation can be inverted.}
\end{proof}

The magog triangles in \eqref{eg-magogs2} are listed in the same order as the
perfect matchings in \cref{fig:TSSCPPsize2} to which they correspond by the bijection in \cref{prop:magog-bij}.

\begin{defn}
An \emph{$(n,k)$-magog trapezoid} is a trapezoidal array 
\[
(m_{i,j})_{1 \leq i \leq n, \max(1,i-k+1) \leq j \leq i}
\]
of positive integers such that 

\begin{itemize}
\item $m_{i,j} \leq m_{i+1,j+1}$ whenever both entries are defined,

\item $m_{i,j} \geq m_{i+1,j}$ whenever both entries are defined,

\item $m_{i,j} \leq j+1$ for all valid $i,j$.
\end{itemize}

\end{defn}

In other words, an $(n,k)$-magog trapezoid is a magog triangle with $n$ rows after removing an equilateral triangle of size {$n-k$} from the bottom left-most side of the magog triangle. 
We let $\mathtt{Magog}(n,k)$ to be the set of $(n,k)$-magog trapezoids.
An example of a configuration in  $\mathtt{Magog}(6,3)$ is 
$$
\begin{array}{ccccccccccccc}
&&&&& 1 \\
&&&&1 && 2  \\
&&&1 && 1 && 2 \\
&&&& 1 &&2 &&4 &&  \\
&&&&& 1 && 2&&5 \\
&&&&&& 1&&3&&7
\end{array}
$$

Zeilberger proved the following result.

\begin{thm}[{\cite[Lemma 1]{Zei96a}}]
\label{thm:Zeil1}
For $n \geq 1$ and $0 \leq k \leq n$, $\mathtt{Magog}(n,k)=\mathtt{Gog}(n,k)$.
\end{thm}

Setting {$k=n$}, we see that the number of PCSMs of size $n$ is equal to the number of TSSCPPs of size $n$, which is another formulation of the ASM theorem~\cite{Zei96a,Kup96}.
The number of these trapezoids for different values of $n$ and $k$ for $n \leq 6$ is given in \cref{tab:num-trapezoids}.
When $k=1$, {it is easy to see that} the sequence gives the well-known Catalan numbers $\frac{1}{n+1} \binom{2n}{n}$.

\begin{table}[htbp!]
\begin{center}
\resizebox{\textwidth}{!}{
\begin{tabular}{c|ccccccccccccc}
$n$ & \\
\hline
1&&&&&& 2\\
2&&&&& 5 && 7\\
3&&&& 14 && 35 && 42\\
4&&& 42 && 219 && 387 && 429\\
5&& 132 && 1594 && 4862 && 7007 && 7436 \\
6& 429 && 12935 && 76505 && 166296 && 210912 && 218348
\end{tabular}
}
\caption{The number of $\mathtt{Magog}(n,k)$ and $\mathtt{Gog}(n,k)$ trapezoids for small values of $n$ and $1 \leq k \leq n$.}
\label{tab:num-trapezoids}
\end{center}

\end{table}

We note that our gog and magog trapezoids, while clearly equinumerous with those defined by Mills--Robbins--Rumsey~\cite{MRR86}, are not exactly the same. However, they are a part of the more general families of such trapezoids defined by Krattenthaler~\cite{krattenthaler-1996,krattenthaler-2016}.

\subsection{The random variables $X^{\mathbf{g}}_n$ and $X^{\mathbf{m}}_n$}
{We now introduce the relevant random variables for PCSMs and TSSCPPs.
For ${C}=({c}_{i,j})_{1 \leq i,j \leq n} \in \mathcal{C}_n$, define 
\begin{equation}
X^{\mathbf{g}}_n \equiv X^{\mathbf{g}}_n(C) =\inf \left\{m \geq 1 \left|
\substack{\ds \mbox{there exists $1 \leq k \leq m$} \\ \ds \mbox{such that }{c}_{k,n-m+k}\not =n } \right. \right\}.
\end{equation}
It is easy to see that $X^{\mathbf{g}}_n \leq n+1$ and that there is only one PCSM $C$ satisfying $X^{\mathbf{g}}_n(C) = n+1$. 
By definition, the PCSMs with $\max T_n=m$ are those that have the $n-m-1$ rightmost northwest-southeast diagonals equal to $n$ and so it follows  that $\max T_n = n-X_n^{\mathbf{g}}$.

Similarly, for $M_n \in \mathcal{M}(G_n^{\mathbf{m}})$, define
\begin{equation}
X^{\mathtt{m}}_n \equiv X^{\mathtt{m}}_n(M_n)= \inf \left\{ m \geq 1 \left|
\substack{((m-1,m),(m-1,m+1)) \mbox{ not covered} \\ \mbox{by a dimer in }M_n  } \right. \right\}.
\end{equation}
{See \cref{fig:tsscppsim} for an illustration.}
Then under the uniform distributions of ASMs and TSSCPPs, $X^{\mathbf{g}}_n$ and $X^{\mathbf{m}}_n$ are random variables.
Using \eqref{eg-pcsm3} and \cref{fig:G_3}, the histogram of values of $X^{\mathbf{g}}_2$ and $X^{\mathbf{m}}_2$ are computed to be 
\[
\begin{array}{|c||c|c|c|}
\hline
\text{Value of } X^{\mathbf{g}}_2 , X^{\mathbf{m}}_2 & 1 & 2 & 3 \\
\hline
\text{Count} & 2 & 4 & 1 \\
\hline
\end{array}
\]
We now show that both have the same distribution in general.
Write
$$
\{X_n^{\mathbf{m}} = m \}=\{ M \in \mathcal{M}(G_n^{\mathbf{m}}): X_n^{\mathbf{m}}(M) = m\}
$$
and
$$
\{X_n^{\mathbf{g}} = m \}= \{ {C} \in \mathcal{C}_n: X_n^{\mathbf{g}}({C}) = m\}.
$$

\begin{prop} 
\label{prop:Zeil2}
For $n\geq 1$ and $1 \leq m \leq n$
\begin{equation}
\mathbb{P}^{\mathbf{m}}_n[ X^{\mathbf{m}}_n=m]
=
\mathbb{P}^{\mathbf{g}}_n[ X^{\mathbf{g}}_n=m].
\end{equation}
\end{prop}

\begin{proof}
The proof of this result relies on a reformulation of $X^{\mathbf{g}}_n$ and $X^{\mathbf{m}}_n$ in terms of gogs and magogs respectively.
We begin with PCSMs.
For the purposes of the proof, let
\begin{equation}
\{Y^{\mathbf{g}}_n = k\} = \bigcup_{j=n+1-k}^{n+1} \{X_n^{\mathbf{g}} = j \} \quad 0 \leq k \leq n,
\end{equation}
and ${\{Y^{\mathbf{g}}_n = -1 \}=\emptyset}$ so that PCSMs $C$ satisfying  
$Y^{\mathbf{g}}_n(C) = m$ have the $n-m$ rightmost northwest-southeast diagonals equal to $n$.
Then, we have 
\begin{equation}
\label{prop:Zeil2:proof:gog1}
\{X_n^{\mathbf{g}} = m \} = {\{Y^{\mathbf{g}}_n = n+1-m \} \setminus \{Y^{\mathbf{g}}_n = n-m\}, \quad 1 \leq m \leq n+1.}
\end{equation}
From the bijection in \cref{prop:pcsm-bij}, PCSMs $C$ satisfying 
$Y^{\mathbf{g}}_n(C) = m$ correspond to ASMs $A =({a}_{i,j})_{1 \leq i,j\leq n+1} \in \mathcal{A}_{n+1}$
satisfying $a_{i,j} = 0$ for $i+j \leq n-m+1$. By a rotation by $180$ degrees, it is clear that the number of such ASMs is the same as those satisfying $a_{i,j} = 0$ for $i+j \geq n+m+1$. 

Now, we use the bijection given after \cref{def:gog-triangle} to convert such an ASM to a gog triangle. In such a gog triangle, the number $k$ must necessarily appear in the first $m+n-k$ rows, for $m+1 \leq k \leq n$.
Therefore, $n$ has to be the last element in the $(m+1)$'th row, $n-1$ and $n$ have to be the last two element in the $(m+2)$'th row, and so on.
Hence, these triangles have a fixed equilateral triangle of size $n-m$ of the form
$$
\begin{array}{ccccccccccc}
&&&& n \\
&&&n-1 && n  \\
&&n-2 && n-1 && n \\
& \iddots &&\iddots &&\ddots &&\ddots &&  \\
m+1 &&m+2 && \cdots  &&  n-1 &&n\\
\end{array}
$$
at the bottom right-most side of the monotone triangle. 
Since this triangle at the bottom right-most side of the monotone triangle are the maximal possible entries (of this part of the monotone triangle), removing this equilateral triangle from the monotone triangle does not change the remaining entries of the monotone triangle and we are left with an $(n,m)$-gog trapezoid. 
Moreover, starting with such a gog-trapezoid, we can reverse the steps to form a PCSM $C$ satisfying $Y^{\mathbf{g}}_n(C) = m$.
Thus, we have shown that 
\begin{equation}
\label{prop:Zeil2:proof:gog2}
|\{Y^{\mathbf{g}}_n = m\}|=|\mathtt{Gog}(n,m)|.
\end{equation}

\begin{figure}[htbp!]
\includegraphics[height=4cm]{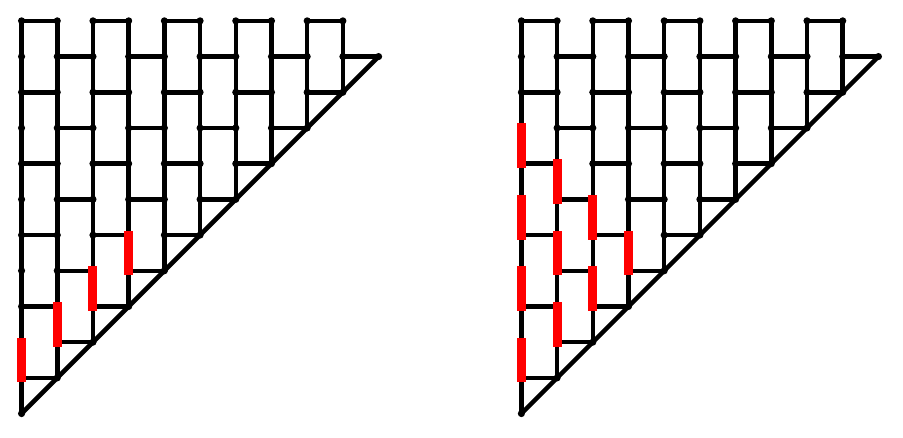}
\caption{The left figure shows the dimers that are fixed in the definition of $Y^{\mathbf{m}}_5(0)$ and the right figure shows the additional edges that are also fixed (ignoring the dimers covering edges incident to the triangle faces).}
\label{fig:TSSCPPfixed}
\end{figure}

We now introduce the corresponding notion for TSSCPPs, that is, 
\begin{equation}
Y^{\mathbf{m}}_n(m)
=
\left\{ M \in \mathcal{M}(G_n^{\mathbf{m}}) \left|
\substack{\ds\mbox{ $\forall 0 \leq k \leq n-1-m$ the edges}  \\
\ds\mbox{$((k,k+1),(k,k+2))$ are covered } \\ 
\ds\mbox{by dimers in the matching $M$} }
\right. \right\}
\end{equation}
for $0 \leq m \leq n$ so that 
$Y^{\mathbf{m}}_n(n)=\mathcal{M}(G_n^{\mathbf{m}})$.
Then, we have for $m\geq 1$
\begin{equation}
\label{prop:Zeil2:proof:magog1}
\{X_n^{\mathbf{m}} = m \}=Y^{\mathbf{m}}_n(n+1-m)\setminus Y^{\mathbf{m}}_n(n-m).
\end{equation}
We next convert magog triangles for the configurations in $Y^{\mathbf{m}}_n(m)$. First, observe that fixing the edges $((k,k+1),(k,k+2))$ to be covered by dimers for all $0 \leq k \leq n-m-1$ forces the edges 
\[
\bigcup_{k=1}^{n-m-1}\bigcup_{r=0}^{n-m-k-1} \{((r,r+2k+1),(r,r+2k+2)) \}
\]
to be covered by dimers; see \cref{fig:TSSCPPfixed} for an example. 

From \cref{prop:tsscpp-bij}, having dimers covering all these edges 
in a \tsscpp {} of size $n$ corresponds to having an equilateral triangle of size $n-m$ with all entries equal to one at the bottom left-most side of the magog triangle. Since the remaining entries of the magog triangle are not restricted by this equilateral triangle of size $m$, such magog triangles correspond exactly to $(n,m)$-magog trapezoids. Hence,
\begin{equation}
\label{prop:Zeil2:proof:magog2}
|Y^{\mathbf{m}}_n(m)|=|\mathtt{Magog}(n,m)|,
\end{equation}
completing the proof using \cref{thm:Zeil1}.
\end{proof}

We have the following theorem that will be proved in \cref{sec:asymptotics}{; see \cref{fig:tsscppsim} for a simulation.}

\begin{thm}
\label{thm:thmtsscpp}
We have that
\begin{equation}
	\lim_{n \to \infty} \mathbb{P}^{\mathbf{m}}_n\left[ \frac{X^{\mathbf{m}}_n- \alpha n}{c_0 n^{1/3}} \geq s \right] = F_1(s).
\end{equation}
\end{thm}

The proof of \cref{thm:mainthm} is now immediate from \cref{thm:thmtsscpp} and \cref{prop:Zeil2}.
\begin{figure}
\begin{center}
\includegraphics[height=7cm]{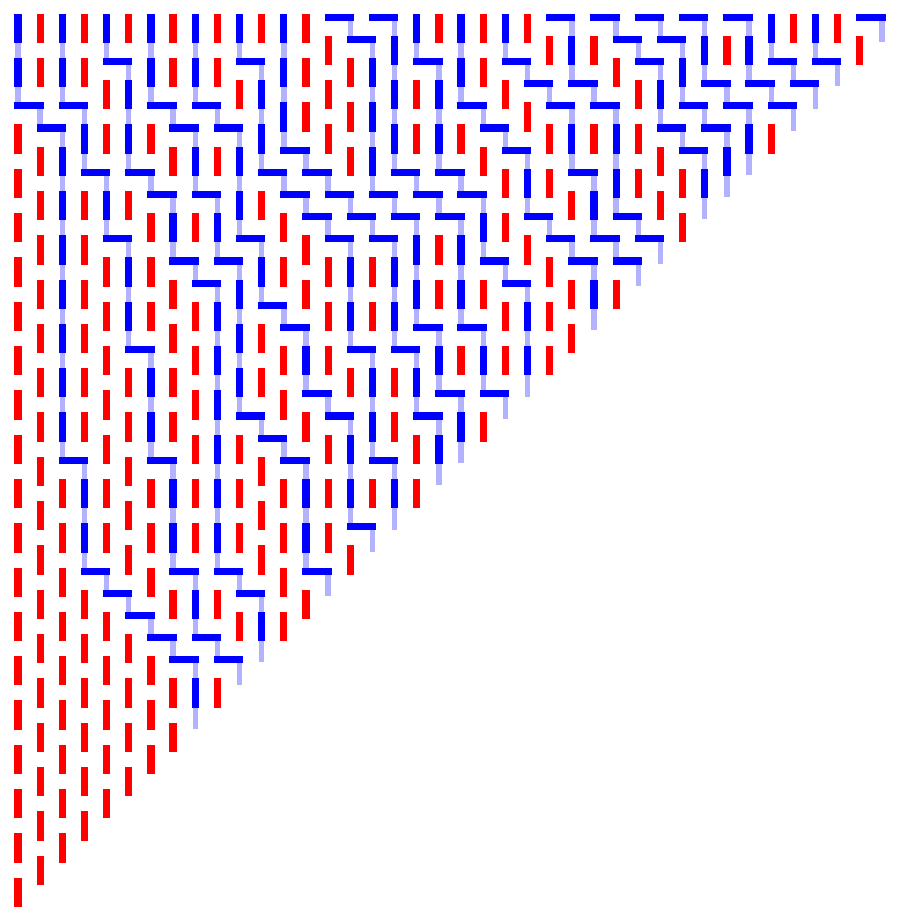}
    \caption{{A simulation of a uniformly random TSSCPP of size 20.  The quantity $X^{\mathbf{m}}_n$ marks the first gap in the red dimers on the bottom diagonal boundary and is equal to $9$ here. Due to an alternate description of \tsscpp{}s using a particular non-intersecting path ensemble~\cite{Andrews-1994}, one can think of this quantity as the hitting point of the leftmost path of this ensemble at a free boundary. We have drawn the paths which are determined by the blue dimers.  
    } }
\label{fig:tsscppsim}
\end{center}
\end{figure}

\section{{GOE Fluctuations for TSSCPPs}} 
\label{sec:asymptotics}

In this section, we provide the proof of \cref{thm:thmtsscpp}.  To do so, we introduce a Pfaffian point process which is determined by certain edges not being covered by dimers on the graph $G_n^{\mathbf{m}}$.  We then state two asymptotic results which allow us to prove \cref{thm:thmtsscpp}. Finally, we prove these two asymptotic results. Since some of the asymptotic computations considered here are relatively standard, where possible, we keep our exposition brief.  

Throughout the remainder of the paper,
we will use the following notation. 
For $a,m \in \mathbb{Z}$ with $m \geq 0$, we denote by $\Gamma_{a,a+1,\dots,a+m}$ a positively oriented contour containing the integers $a$ through to $a+m$ and no other integers.  In particular, $\Gamma_0$ is a positively oriented circle around the origin with radius less than 1.

\subsection{Pfaffian point process and formulas}
We define a simple point process on $\mathcal{L}_n =[0,{n-1}] \cap \mathbb{Z}$ using the dimer model on $G_n^{\mathbf{m}}$ in the following way:  associate a particle at $k \in \mathcal{L}_n$ if and only if the edge $((k,k+1),(k,k+2))$ in $G^\mathbf{m}_n$ is not covered by a dimer. Let $\{z_i\}$ denote this point process on $\mathcal{L}_n$. Note that we have not defined the point process on edges $((k,k+1),(k,k+2)) \in  G^\mathbf{m}_n$ for ${n} \leq k \leq 2n-1$. \cref{fig:pp} shows an example. 

\begin{figure}
\includegraphics[height=5cm]{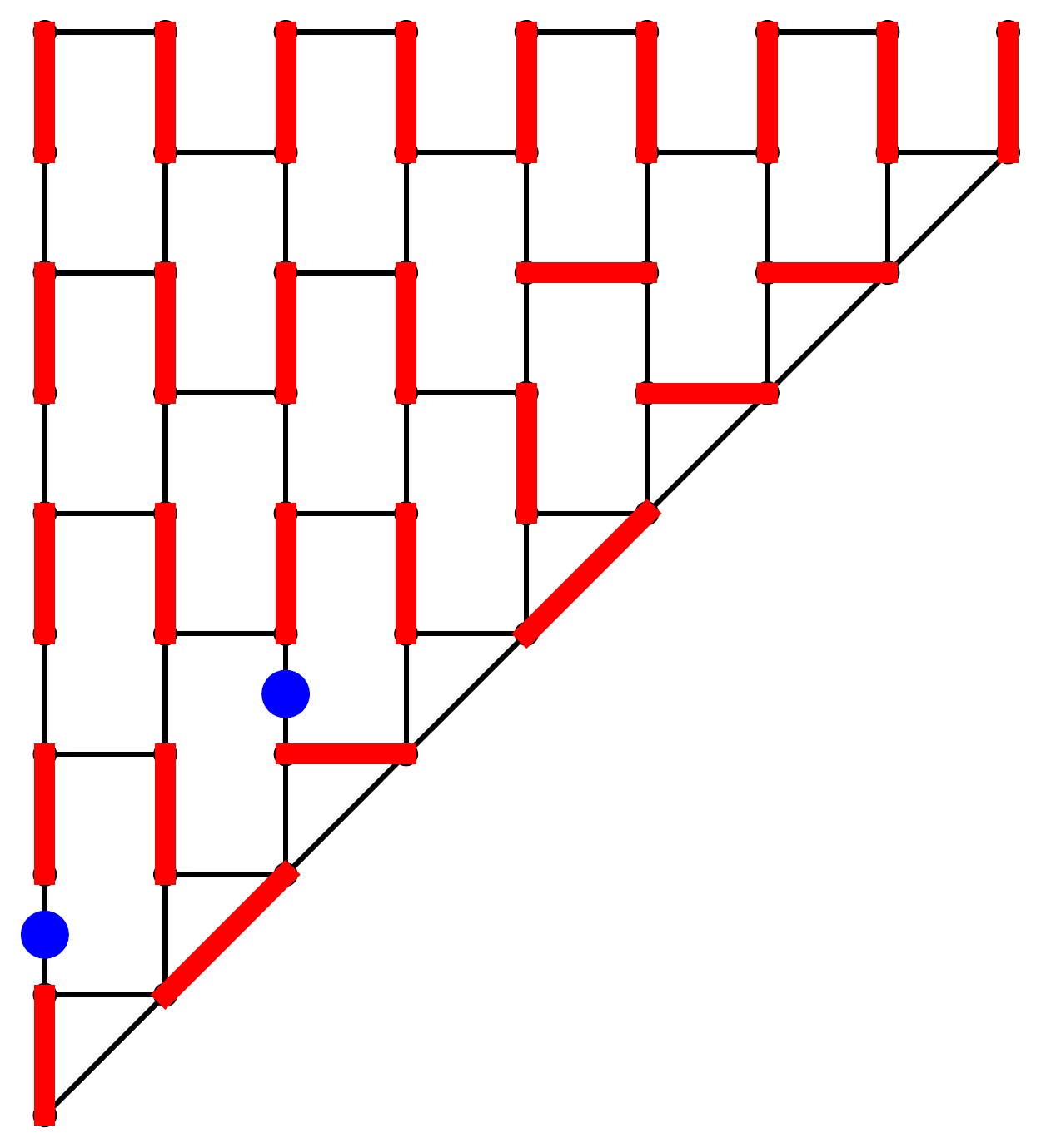}
\caption{The particle process where the blue dots denote the particles sitting on particular edges not covered by dimers. }
\label{fig:pp}
\end{figure}

 To describe this point process, we introduce the following formulas: {for $x \in \mathcal{L}_n$,} let
\begin{equation}\label{eq:Goverline}
\overline{G}(n,\ell,x)= \frac{1}{2 \pi \mathrm{i}} \int_{\Gamma'_{\{x\}}}\text{d}w \; G_{n,\ell,x}(w)
\end{equation}
where $\Gamma'_{\{x\}}= \Gamma_{0,1,\dots,\lfloor \frac{x}{2} \rfloor}$,
\begin{equation}\label{eq:Gnlx}
	G_{n,\ell,x}(w)= \frac{S(n-\ell,w-\ell,x-2\ell)}{{ (n-w+1)\prod_{j=\ell}^n(w-j)}} = G_{n-\ell,0,x-2\ell}(w-\ell)
\end{equation}
and 
\begin{equation}\label{eq:Snwx}
\begin{split}
S(n,w,x)&= \frac{\Gamma(n+w+2)\Gamma(2n-w+2)\Gamma(2n-2w+1)(3n-3w+2)}{\Gamma(n+w-x+1)\Gamma(3n-w+3)\Gamma(x-2w+1)} \\
\end{split}
\end{equation}
{To see the poles and zeros of $S(n,w,x)/\prod_{j=0}^n(w-j)$, notice that
\begin{equation}
\begin{split}\label{eq:Snwxpoles}
\frac{S(n,w,x)}{\prod_{j=0}^n(w-j)}&= (3n-3w+2) \frac{\prod_{j=0}^x (n+w+1-j) \prod_{j=0}^{2n-x-1}(2n-2w-j)}{\prod_{j=0}^n(3n-w+2-j)\prod_{j=0}^n (w-j)} \\
&=(3n-3w+2) \frac{\prod_{j=-n-1}^{x-n-1} (w-j) \prod_{j=-2n}^{-x-1}(-2w-j)}{\prod_{j=-3n-2}^{-2n-2}(-w-j)\prod_{j=0}^n (w-j)} \\
&=(3n-3w+2)(-1)^{n-x+1}2^{2n-x} \frac{\prod_{j=-n-1}^{x-n-1} (w-j) \prod_{j=x+1}^{2n}(w-j/2)}{\prod_{j=2n+2}^{3n+2}(w-j)\prod_{j=0}^n (w-j)}.
\end{split}
\end{equation}
}
Further, {for $x,y \in \mathcal{L}_n$,} let
\begin{equation}\label{eq:Hnlx}
\overline{H}(n,\ell,x)= \frac{1}{2 \pi \mathrm{i}} \int_{\Gamma'_{\{x\}}} \text{d} w \; G_{n,\ell,x}(w) \frac{n-w+1}{n+w-x},
\end{equation}
and define
\begin{equation}\label{eq:f11}
f^{1,1}(x,y)=\sum_{\ell=0}^\kappa\big( \overline{G}(n,\ell+1,x)\overline{G}(n,\ell,y)- \overline{G}(n,\ell,x)\overline{G}(n,\ell+1,y)\big),
\end{equation}
\begin{equation}\label{eq:f22}
\begin{split}
f^{2,2}(x,y)=& \mathrm{sgn}(x-y)+\sum_{\ell=0}^\kappa \sum_{m_1=0}^x
\sum_{m_2=0}^y \Big(\overline{H}(n,\ell+1,x-m_1)\overline{H}(n,\ell,y-m_2)\\ 
&- \overline{H}(n,\ell,x-m_1)\overline{H}(n,\ell+1,y-m_2) \Big) \\
&-(-1)^n\sum_{m=0}^x \overline{H}(n,0,x-m)
+(-1)^n\sum_{m=0}^x \overline{H}(n,0,y-m),
\end{split}
\end{equation}
\begin{equation}\label{eq:f12}
\begin{split}
f^{1,2}(x,y)=& (-1)^n \overline{G}(n,0,x)-\sum_{\ell=0}^\kappa\sum_{m=0}^y\Big( \overline{G}(n,\ell+1,x)\overline{H}(n,\ell,y-m)\\&- \overline{G}(n,\ell,x)\overline{H}(n,\ell+1,y-m)\Big),
\end{split}
\end{equation}
and
\begin{equation}\label{eq:f21}
f^{2,1}(x,y)=-f^{1,2}(y,x),
\end{equation}
with $\kappa$ taken to be such that 
\begin{equation}\label{eq:kappa}
\min \bigg(\bigg[\frac{x}{2}\bigg],\bigg[\frac{y}{2}\bigg]\bigg) <\kappa<\bigg[\frac{n+1}{2}\bigg].
\end{equation}
The exact value of $\kappa$ is unimportant since  $\overline{G}(n,\ell,x)=\overline{H}(n,\ell,x)=0$ for $\ell > \big[\frac{x}{2}\big]$, which follows from the fact that there are no poles of $G_{n,\ell,x}(w)$ inside $\Gamma_{\{x\}}'$ for $\ell>\big[\frac{x}{2}\big]$.   
We now have the following important result.

\begin{prop}
\label{prop:discrete}
The point processes $\{z_i\}$ is a Pfaffian point process with correlation kernel $\mathbf{f}:\mathcal{L}_n \times \mathcal{L}_n \to \mathbb{R}$  given by
\begin{equation}
\mathbf{f}(x,y)=\left( \begin{array}{cc}
f^{1,1}(x,y) & f^{1,2}(x,y) \\
f^{2,1}(x,y) & f^{2,2}(x,y) \end{array} \right)
\end{equation}
for $x,y \in \mathcal{L}_n$ with $f^{1,1},f^{2,2},f^{1,2},f^{2,1}$ given~\eqref{eq:f11},~\eqref{eq:f22},~\eqref{eq:f12}, and~\eqref{eq:f21}.
In other words, since $\{z_i\}$ is a discrete simple point process, we have 
\begin{equation}
\mathbb{P}^{\mathbf{m}}_n[\mbox{Particles at }x_1,\dots, x_m]=\mathrm{Pf}[\mathbf{f}(x_i,x_j)]_{1 \leq i,j \leq m},
\end{equation}
where $x_1,\dots, x_{{m}}$ are distinct points in $\mathcal{L}_n$.  
\end{prop}

The proof of \cref{prop:discrete} is given in \cref{sec:derivation}.

\subsection{Proof of \cref{thm:thmtsscpp}}

We have the following asymptotic result.

\begin{prop}
\label{prop:kernelconv}
Uniformly for $\xi,\eta$ in a compact set and $1 \leq i,j\leq 2$, 
\begin{equation}
\lim_{n \to \infty} (c_0 n^{\frac{1}{3}})^{4-i-j} f^{i,j}([\alpha n-c_0 n^{\frac{1}{3}} \xi],[\alpha n-c_0 n^{\frac{1}{3}} \eta ]) =K_{\mathrm{GOE}}^{ij}(\xi,\eta),
\end{equation}
where $\alpha$ and $c_0$ are defined in~\eqref{eq:scalings} and the formulas for $K_{\mathrm{GOE}}^{ij}$ are given in \cref{subsec:TWGOE}.
\end{prop}

We also need the following bounds.

\begin{prop} 
\label{prop:kernelbounds}
Fix $M>0$. For $\xi,\eta {<}-M$, there exists constants $C,c>0$ that depend on $M$ such that
$$(c_0 n^{\frac{1}{3}})^2f^{1,1}([\alpha n-c_0 n^{\frac{1}{3}} \xi],[\alpha n-c_0 n^{\frac{1}{3}} \eta ]) \leq C e^{-c(\xi+\eta)},$$
$$c_0 n^{\frac{1}{3}}f^{1,2}([\alpha n-c_0 n^{\frac{1}{3}} \xi],[\alpha n-c_0 n^{\frac{1}{3}} \eta ]) \leq C e^{-c\xi},$$
and
$$f^{2,2}([\alpha n-c_0 n^{\frac{1}{3}} \xi],[\alpha n-c_0 n^{\frac{1}{3}} \eta ]) \leq C. $$
\end{prop}

Both results are proved in \cref{subsec:asymptotics}.  We now prove \cref{thm:thmtsscpp}. 

\begin{proof}[Proof of \cref{thm:thmtsscpp}]
We first use inclusion-exclusion to write that 
	\begin{equation}
		\mathbb{P}^{\mathbf{m}}_n[X_n^{\mathbf{m}} \geq s]=\mathbb{P}^{\mathbf{m}}_n[\mbox{No particles in $[0,s)$}]=\mathrm{Pf}[\mathbb{J}-\mathbf{f}]_{\ell^2(0,s)}
	\end{equation}
We reverse the coordinates by considering $x \in \mathcal{L}_n$ to be of the form $x=\alpha n - c_0 n^{\frac{1}{3}} \xi$ where $\xi$ is our new variable.  
The bounds given in \cref{prop:kernelbounds} are precisely those needed to give an integrable upper bound in the expansion of the Fredholm Pfaffian; see for instance \cite[Lemma 2.5]{BBCS:18}. This allows us to use the dominated convergence theorem to pass limits through the integrals and use \cref{prop:kernelconv}, which then shows convergence to the GOE Tracy--Widom distribution.  
\end{proof}

\subsection{Asymptotics}\label{subsec:asymptotics}

In this section, we give the proofs of \cref{prop:kernelconv} and \cref{prop:kernelbounds}.  Before doing so, we will need some preliminary results.  {Throughout, we will use that the logarithm takes its branch cut on the negative axis.}

For $a \in [0,\alpha], X,\lambda \in \mathbb{R}, r\in \{0,1\},w \in \mathbb{C}$, introduce the functions
\begin{equation}\label{eq:S1}
\begin{split}
	S_1(w,a)=& (1+w) \log (1+w) +(2-w)\log (2-w) +(2-2w)\log(2-2w) \\
	&-w \log (-w) -(1-w)\log(1-w) -(1+w-a)\log(1+w-a)\\
	 &-(3-w)\log (3-w)-(a-2w)\log(a-2w),
\end{split}
\end{equation}
\begin{equation}\label{eq:S2}
\begin{split}
S_2(w;a,\lambda,X)=& -X\log(1+w-a)+X\log(a-2w) -2\lambda \log(1+w) \\
& +2\lambda \log(3-w) -\lambda \log(2-w) + \lambda \log (-w),
\end{split}
\end{equation}
and
\begin{equation}\label{eq:S3}
\begin{split}
S_3(w;a,r)=& 2\log(1+w) +2 \log(2-w) +\log (2-2w)-\log(1+w-a) \\
&-3 \log (3-w) -\log(a-2w)-\log(1-w)-1-2r \log (1+w)\\
& {+}2r \log (3-w) -r \log(2-w)+r \log(-w).
\end{split}
\end{equation}
Let 
\begin{equation}
	\begin{split}\label{eq:An}
&A_n(w,a,\lambda,X,r)=n \frac{3-3w+\frac{2}{n}}{1-w+\frac{1}{n}} 
\Bigg[ \frac{(1+w-\frac{an-Xn^{\frac{1}{3}}}{n}+\frac{1}{n})}{(1+w-\frac{2(\lambda n^{\frac{1}{3}}+r)}{n}+\frac{2}{n})}
\\
&\times  \frac{(3-\frac{2(\lambda n^{\frac{1}{3}}+r)}{n}-w+\frac{3}{n})(\frac{an-Xn^{\frac{1}{3}}}{n}-2w+\frac{1}{n})(1+\frac{1}{n}-w)}
{(2-\frac{(\lambda n^{\frac{1}{3}}+r)}{n}-w+\frac{2}{n})(2-2w+\frac{1}{n})(\frac{\lambda n^{\frac{1}{3}}+r}{n}-w)}\Bigg]^{\frac{1}{2}},
\end{split}
\end{equation}
and take
\begin{equation}\label{eq:Stilde}
\tilde{S}_n(w,a,\lambda,X,r)=S_1(w,a)+n^{-\frac{2}{3}} S_2(w,a,\lambda,X)+n^{-1}S_3(w,a,r)-\frac{\log n}{n}+\frac{1}{n}.
\end{equation}

The following lemma heavily simplifies our computations below. 

\begin{lem}\label{lem:changeofvariables}
\begin{equation}
\overline{G}(n,\ell,x)=\overline{G}(n-\ell,0,x-2\ell)
\end{equation}
and
\begin{equation}
\overline{H}(n,\ell,x)=\overline{H}(n-\ell,0,x-2\ell)
\end{equation}

\end{lem}
\begin{proof}
The result follows immediately for both equations by taking the change of variables $w\mapsto w+\ell$ in both~\eqref{eq:Gnlx} and~\eqref{eq:Hnlx} as $G_{n,\ell,x}(w)$   and the contour $\Gamma_{\{x\}}'$
get mapped to $G_{n-\ell,0,x-2\ell}(w)$ and  $\Gamma_{\{x-2\ell\}}'$ respectively under this change of variables.
\end{proof}

The next lemma gives a clearer form for the asymptotics of $\overline{G}$ and $\overline{H}$. 

\begin{lem}\label{lem:startingasymp}
For $a \in [0,\alpha], X,\lambda \in \mathbb{R}, r\in \{0,1\}$ 
\begin{equation}\label{lemeq:startingasympG}
\begin{split}
\overline{G}(n, & [\lambda n^{\frac{1}{3}}] +r, [an-Xn^{\frac{1}{3}}])\\
&=\frac{(-1)^{n-[\lambda n^{\frac{1}{3}}] +r+1}}{2\pi \mathrm{i}} \int \text{d} w \; A_n(w,a,\lambda,X,r)e^{n\big(\tilde{S}_n(w,a,\lambda,X,r)+O\big(\frac{1}{n^2}\big)\big)},
	\end{split}
\end{equation}
and
\begin{equation}\label{lemeq:startingasympH}
\begin{split}
	&\overline{H}(n,[\lambda n^{\frac{1}{3}}] +r,[an-Xn^{\frac{1}{3}}])
	=\frac{(-1)^{n-[\lambda n^{\frac{1}{3}}] +r+1}}{2\pi \mathrm{i}} \\
&\times \int \text{d} w \, A_n(w,a,\lambda,X,r)e^{n\big(\tilde{S}_n(w,a,\lambda,X,r)+O\big(\frac{1}{n^2}\big)\big)}\frac{n-nw+1}{n+nw-an-Xn^{\frac{1}{3}}}, 
\end{split}
\end{equation}
	where in both equations the contour is positively oriented surrounding the  points $\{0,\frac{1}{n},\frac{2}{n},\dots, \frac{1}{n}\big[\frac{1}{2}(an-Xn^{\frac{1}{3}}\big]\}$ and no other integer points divided by $n$, and $\tilde{S}_n$ and $A_n$ are defined in~\eqref{eq:S1} and~\eqref{eq:An} respectively.
\end{lem}

\begin{proof}
	We only consider the first equation as the second equation is analogous. Write $\bar{X}=[an-Xn^{\frac{1}{3}}]$ and $\ell=[\lambda n^{\frac{1}{3}}]+r$. We have by~\eqref{eq:Gnlx} 
\begin{equation}
	\begin{split}
        &\overline{G}(n,\ell,\bar{X})= \frac{1}{2\pi \mathrm{i}}\int_{\Gamma_{\{\overline{X}\}}'}\text{d}w \; \frac{ \Gamma(n-2\ell+w+2)\Gamma(2n-\ell-w+2) }{\Gamma(n+w-\bar{X}+1) \Gamma(3n-2\ell-w+{3}) }\\
		&\times \frac{\Gamma(2n-2w+1) (3n-3w+2)}{ \Gamma(\bar{X}-2w+1)(n-w+1) \prod_{j=\ell}^{n} (w-j) }.
	\end{split}
	\end{equation}
Next, notice that
	\begin{equation}
\frac{1}{\prod_{j=\ell}^{n}(w-j)}=(-1)^{n-\ell+1} \frac{1}{\prod_{j=0}^{n-\ell}(\ell+j-w)}
=(-1)^{n-\ell+1}\frac{\Gamma(\ell-w)}{\Gamma(n+1-w)}
\end{equation}
which is analytic as we deform the contour to be away from the integer points.  We now apply  the change of variables $w \mapsto w n$ which gives
\begin{equation}
	\begin{split}
		\frac{(-1)^{n-\ell+1}}{2\pi \mathrm{i}}\int& \text{d}w \; n \frac{ \Gamma(n-2\ell+nw+2)\Gamma(2n-\ell-nw+2) \Gamma(2n-2nw+1)}{ \Gamma(n+nw-\bar{X}+1) \Gamma(3n-2\ell-nw+{3}) \Gamma(\bar{X}-2nw+1)}  \\ 
		&\times \frac{\Gamma(\ell-nw)(3n-3nw+2)}{\Gamma(n+1-nw)(n-nw+1)},
	\end{split}
\end{equation}
	where the contour is the same as the one described in the statement of the lemma.  We now apply Stirling's approximation; see \cite[{Lemma} 7.3]{Pet14} for the precise statement.  After expanding the exponent into terms of order $n$, $n^{\frac{1}{3}}$ and constant order, we arrive at the first equation.  The second equation follows in a similar fashion. 
\end{proof}

	The leading order asymptotics for $\overline{G}$ and $\overline{H}$ in~\eqref{lemeq:startingasympG} and~\eqref{lemeq:startingasympH} is from $S_1(w,a)$ defined in \eqref{eq:S1}. The next four lemmas focus on this function.

	\begin{lem}
		We have that $\frac{\partial S_1}{\partial w} (w_{\pm}(a),a)=0$  where
		\begin{equation}
			w_{\pm}(a)= \frac{ (4-10a+a^2)\pm 2(a-2) \sqrt{1-4a+a^2}}{-8-4a+a^2}
		\end{equation}
		and that $w_+(\alpha)=w_-(\alpha)$ where $\alpha$ is given in ~\eqref{eq:scalings}.  For $0 \leq a < \alpha$, we have that $w_-(a)\leq w_+(\alpha) \leq w_+(a)\leq 0$ while for {$ \alpha<a<2+\sqrt{3}$}, $w_{\pm}(a) \in \mathbb{C}$.  
	\end{lem}
	\begin{proof}
We have that 
		\begin{equation}
			\frac{\partial S_1}{\partial w}(w,a)= \log \left[ \frac{ (1+w)(1-w)(3-w)(a-2w)^2}{(2-w)(2-2w)^2(-w)(1+w-a)} \right] 
		\end{equation}
		and so $\frac{\partial S_1}{\partial w}(w,a) =0$ means that 
		$$
		(1+w)(3-w)(a-2w)^2=4(2-w)(1-w)(-w)(1+w-a).
		$$
		Solving for $w$ gives the formulas for $w_\pm (a)$. It follows immediately that $w_+(\alpha)=w_-(\alpha)$.  To see that $w_+(\alpha) \leq w_+(a) \leq 0$, differentiate the formula for $w_+(a)$ given in the statement of the lemma with respect to $a$ which gives
		\begin{equation}
			\frac{d}{da} w_+(a) = \frac{ 6(16-4a+a^2)\sqrt{1-4a+a^2} -6(16-28a+7a^2)}{(-8-4a+a^2)^2 \sqrt{1-4a +a^2}}.
		\end{equation}
		Multiplying the numerator by $(16-4a+a^2)\sqrt{1-4a+a^2} +(16-28a+7a^2) \geq 0$ for $0<a\leq \alpha$, we see that $\frac{d}{da} w_+(a) \leq 0$ for $0<a\leq \alpha$.  Evaluating $w_+(a)$ at $a= \alpha$ and $a=0$, we see that $w_+(\alpha) \leq w_+(a) \leq 0$. A similar argument holds for $w_-(a)$.  
	\end{proof}

	\begin{lem}
	\label{lem:D2S1} 
	We have that 
		\begin{equation} 
			\begin{split}
				&\frac{\partial^2 S_1}{\partial w^2} (w_{\pm}(a),a)\\=&\frac{{-}1}{6(4-a)^2(2-a)a^2} \Big( -256+1296a-1428a^2+616a^3-117a^4\\
			&+12a^5-a^6 \pm (-256+752a-396a^2+104a^3-13a^4)\sqrt{1-4a+a^2} \Big),
			\end{split}
		\end{equation}
		\begin{equation}
			\frac{ \partial^2 S_1}{\partial w^2} (w_+(a),a)>0 \mbox{ and }	\frac{ \partial^2 S_1}{\partial w^2} (w_-(a),a)<0 \mbox{ for }0<a<\alpha.
		\end{equation}
	We also have that  $\frac{ \partial^2 S_1}{\partial w^2} (w_+(a),a)$
		(resp. $\frac{ \partial^2 S_1}{\partial w^2} (w_-(a),a)$) is decreasing (resp. increasing) for $0<a<\alpha$. 

	\end{lem}

	\begin{proof}
		The formula for $\frac{ \partial^2 S_1}{\partial w^2} (w_+(a),a)$ follows from differentiating $S_1$ twice, setting $w=w_\pm (a)$ and simplifying using computer algebra.  Since we have an  explicit formula for $\frac{ \partial^2 S_1}{\partial w^2} (w_+(a),a)$, we can differentiate with respect to $a$ and set to zero, we find no solutions in $0<a< \alpha$.  This means that either $\frac{ \partial^2 S_1}{\partial w^2} (w_+(a),a)$ increasing or decreasing for all $0< a< \alpha$. As $a$ tends to zero, we see that $\frac{ \partial^2 S_1}{\partial w^2} (w_+(a),a)$ tends to $+\infty$  and as $a$ tends to $\alpha$, we see that $\frac{ \partial^2 S_1}{\partial w^2} (w_+(a),a)$ tends to $0$.  This means that $\frac{ \partial^2 S_1}{\partial w^2} (w_+(a),a)$ is decreasing and $\frac{ \partial^2 S_1}{\partial w^2} (w_+(a),a)>0$ for $0<a<\alpha$. An analogous argument holds for $\frac{ \partial^2 S_1}{\partial w^2} (w_-(a),a)$.
	\end{proof}

	\begin{lem}\label{lem:D3S1}
		For $\alpha$ defined in~\eqref{eq:scalings}, we have
		$$
		\frac{\partial^3 S_1}{\partial w^3} (w_\pm(\alpha), \alpha) =\frac{81}{4}.
		$$
	\end{lem}
	\begin{proof}
	This immediately follows from differentiating $S_1$ three times and setting $w=w_+(\alpha)=w_-(\alpha)$. 
	\end{proof}

We can now give the descent contours for $S_1(w,a)$ which will be used in the asymptotic analysis below. We only need to describe these contours for $w \in \mathbb{H}$, the upper half plane, since $S_1(\overline{w},a)=\overline{S_1({w},a)}$ for $w \in \mathbb{C}$. \cref{fig:contours} shows the steepest descent and ascent contours. 
\begin{figure}
\includegraphics[height=6cm]{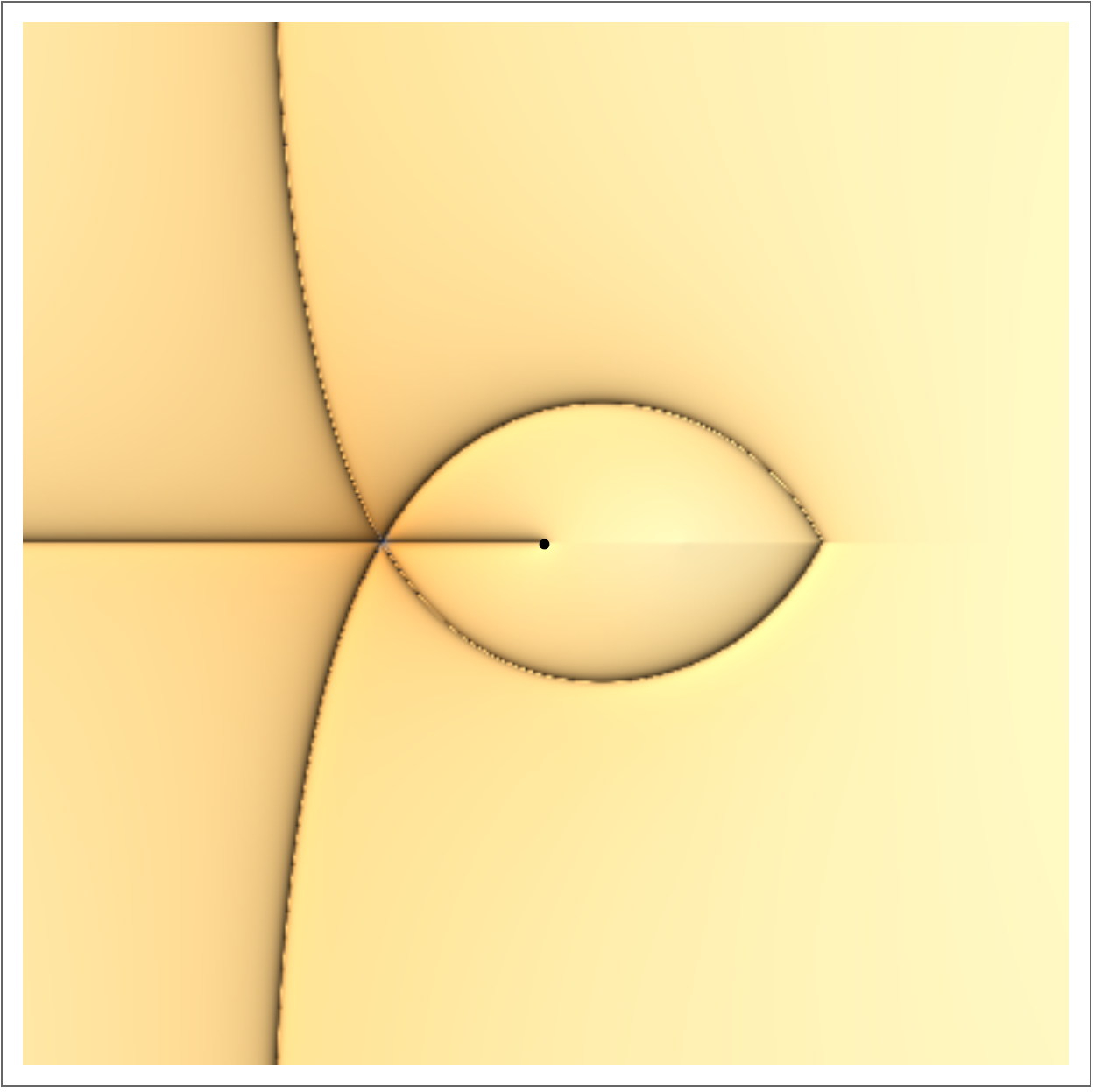}
\ \
\includegraphics[height=6cm]{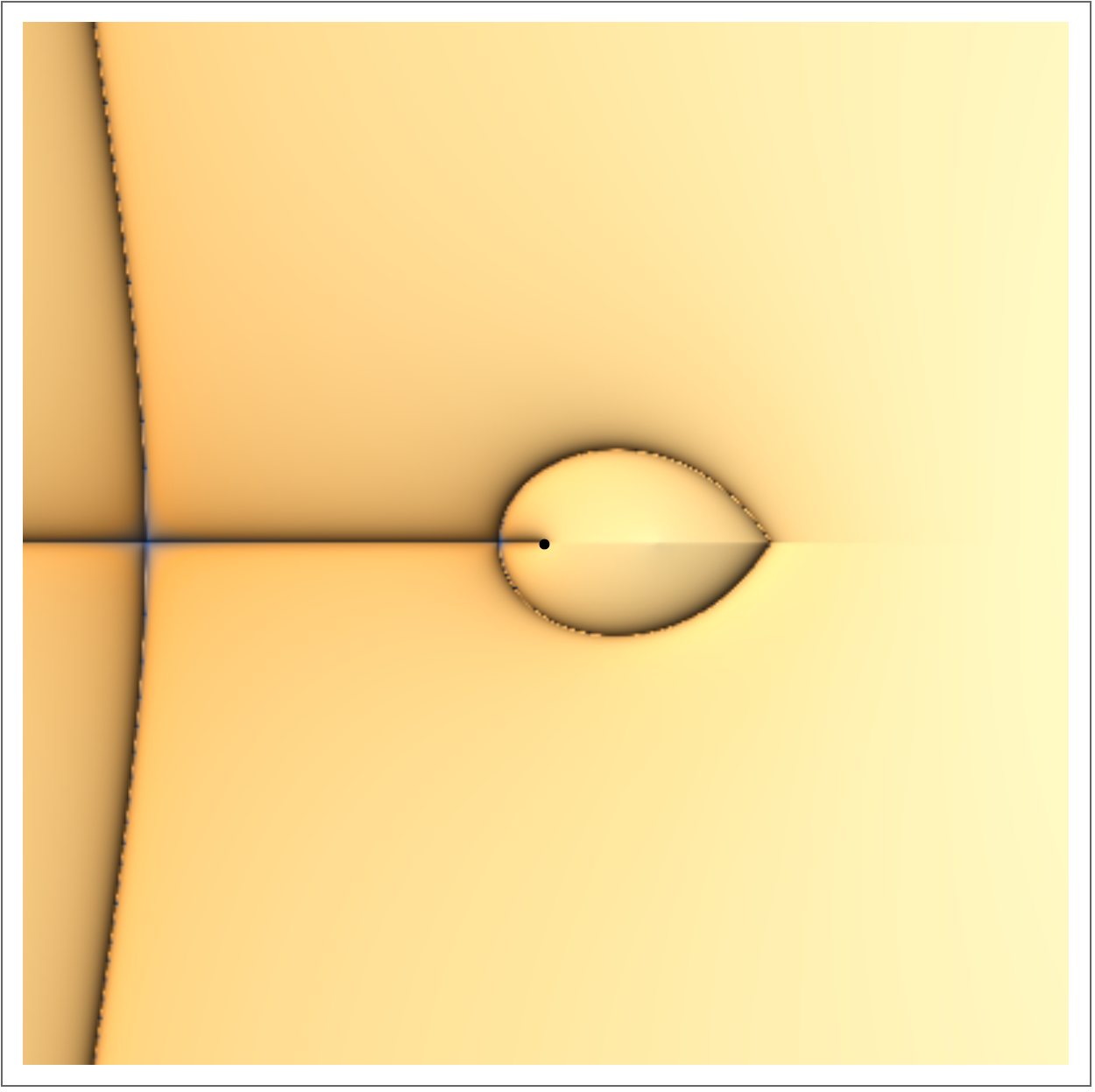}
	\caption{The left figure shows the steepest ascent and descent contours for $a=
	\alpha$, while the right figure shows the steepest ascent and descent contours for $a<\alpha$. The black dot in each of the figures represents the origin. }
\label{fig:contours}

\end{figure}

\begin{lem} \label{lem:steepestdescentcontour}
	For $0<a<\alpha$, the steepest descent contour for $S_1(w,a)$ leaves $w_+(a) \in (w_+(\alpha),0)$ at angle $\pi/2$ and ends at $a$.  The steepest descent contour for $S_1(w,\alpha)$ leaves $w_+(\alpha)$ at angle $\pi/3$ and ends at $\alpha$.
\end{lem}
From the above lemma, the steepest descent contour in the upper and lower half plane forms a closed curve in $\mathbb{C}$.  
We label this positively oriented closed curve $\gamma_a$ for $0<a \leq \alpha$.  

\begin{proof}
	A computation shows that { $w_+(a)+1-a>0$ } so  $\mathrm{Im} S_1(w_+(a),a)=0$.  The steepest descent contour is the level line of $\mathrm{Im} S_1(w,a)=0$ which we need to describe.  From \cref{lem:D2S1}, we see that the contour leaving $w_+(a)$ at angle $\pi/2$ is the steepest descent contour for $S_1(w,a)$ with $0<a<\alpha$ and from \cref{lem:D3S1}, we see that the contour leaving $w_+(\alpha)$ at angle $\pi/3$ is the steepest descent contour for $S_1(w,\alpha)$.  

	We can also write $S_1(w,a)$ as 
		\begin{equation}
			a \pi \mathrm{i}{\mathrm{Im}(w)} +(2-a)\log 2 + \left( \int_{-1}^{a-1}+\int_{\frac{a}{2}}^1-\int_2^3-\int_0^{\frac{a}{2}} \text{d}t \; \log(w-t) \right),
		\end{equation}
{which follows by integrating out and noting that we have used the principal branch of the logarithm.}
From this expression, we see that {
\begin{equation}
\mathrm{Re} S_1(w,a)=(2-a)\log 2 + \left( \int_{-1}^{a-1}+\int_{\frac{a}{2}}^1-\int_2^3-\int_0^{\frac{a}{2}} \text{d}t \; \log|w-t| \right),
\end{equation}
and so}
\begin{equation}
	\mathrm{Re} S_1(w,a)={a} \log |w|+{ (2-a)\log 2}+O\bigg( \frac{1}{|w|} \bigg), 
\end{equation}
		as $|w| \to \infty$.  This implies that the steepest descent contour is contained in some ball with finite radius, that is, the steepest descent contour will end on the real axis.  To find this point, we evaluate $\mathrm{Im}S_1(t,a)$  for $t\in \mathbb{R}$ and find that 
		\begin{equation}
			\mathrm{Im}S_1(t,a)= 
			\begin{cases}
				a\pi & t \leq -1, \\
				-(1+t-a)\pi & -1 < t \leq -1+a,\\
				0 & -1+a<t \leq 0, \\
			{	-t \pi} & 0<t \leq \frac{a}{2}, \\
				{(t-a)\pi }& \frac{a}{2}<t\leq 1, \\
				{(1-a) \pi} & 1<t \leq 2, \\
				{(3-a-t)\pi} & 2 < t \leq 3,  \\
			{-a \pi} & t>3. 
			\end{cases}
		\end{equation}
		From the above equation, the only possible ending of the steepest descent contour starting from $w_+(a)$ is $a$ for $0 < a \leq \alpha$.  
\end{proof}

We need the following estimates.
\begin{lem}
\label{lem:estimates}
	Suppose that $M>0$ is fixed and let $m \in \mathbb{N}$. For $v{<}-M$, we have  
	\begin{equation}
		|\overline{G}(m,0,\alpha m-vm^{1/3})|\leq C e^{-c v},
	\end{equation}
	and
	\begin{equation}
		|\overline{H}(m,0,\alpha m-vm^{1/3})|\leq C e^{-c v},
	\end{equation}
where $C,c>0$ are constants that only depend on $M$.  
\end{lem}

The proof is postponed until after the proof of \cref{prop:kernelconv}, which we give now.

\begin{proof}[Proof of \cref{prop:kernelconv}]
	Fix $\delta >0$. We first demonstrate 
\begin{claim}\label{claim}
 For $0 \leq \lambda, \xi,\mu \leq n^{\delta}$, we have
	\begin{equation}\label{propproof:overlineGasymp}
		\begin{split}
			&\overline{G}(n,[\lambda n^{\frac{1}{3}}]+r,[\alpha n- \xi c_0 n^{\frac{1}{3}}]) = \frac{(-1)^{n+[\lambda n^{\frac{1}{3}}]+r}(1+O(n^{-\frac{1}{3}}))}{2 \pi \mathrm{i}} \\
			& \times \int_{\Gamma_{\mathrm{Ai}}} \text{d} \omega \; c_1 \frac{3 \sqrt{3}}{2 n^{\frac{1}{3}}} (1-9c_1r \omega n^{-\frac{1}{3}})e^{\frac{\omega^3}{3} -\xi \omega - 9c_1 \lambda \omega },
		\end{split}
	\end{equation}
and
\begin{equation}\label{propproof:overlineHasymp}
		\begin{split}
			&\overline{H}(n,[\lambda n^{\frac{1}{3}}]+r,[\alpha n- \xi c_0 n^{\frac{1}{3}}-\mu n^{\frac{1}{3}}]) = \frac{(-1)^{n+[\lambda n^{\frac{1}{3}}]+r} (1+O(n^{-\frac{1}{3}}))}{2 \pi \mathrm{i}} \\ & \times \int_{\Gamma_{\mathrm{Ai}}} \text{d} \omega \; c_1 \frac{3 \sqrt{3}}{ n^{\frac{1}{3}}} (1-9c_1r \omega n^{-\frac{1}{3}})e^{\frac{\omega^3}{3} -(\xi+\mu c_1 3 \sqrt{3}) \omega - 9c_1 \lambda \omega },
		\end{split}
	\end{equation}
	where $c_0$ and $c_1$ are given in~\eqref{eq:scalings} and $r \in \{0,1\}$.
\end{claim}

	Once we have proved the above claim, we will be able to compute the asymptotics for $f^{1,1},f^{1,2}$ and $f^{2,2}$ since for any $\lambda, \xi$ or $\mu$ being larger than $n^{\delta}$, we can apply \cref{lem:changeofvariables}, {which together} with \cref{lem:estimates} gives exponential decay seen in \cref{prop:kernelbounds}.

\begin{proof}[Proof of \cref{claim}]
	We begin with showing~\eqref{propproof:overlineGasymp} with our starting point given by the right side of~\eqref{lemeq:startingasympG} in \cref{lem:startingasymp}, but setting $a=\alpha$ and $X=\xi c_0$. The computation proceeds using a  standard saddle point approximation argument so we will only give the main steps. From the formula \eqref{eq:Stilde} for $\tilde{S}_n(w,\alpha,\lambda,c_0 \xi,r)$, the main asymptotic contribution for the integral on the right side of~\eqref{lemeq:startingasympG} is given by $S_1(w,\alpha)$, which is defined in~\eqref{eq:S1}. We deform the contour to $\gamma_\alpha$  described just after \cref{lem:steepestdescentcontour}.  
{We split $\gamma_\alpha$ into two parts; for $\varepsilon>0$, let $\mathcal{G}_{w_+(\alpha),\varepsilon}=\gamma_\alpha \cap B(w_+(a),\varepsilon)$ and $\mathcal{G}_{w_+(a),\varepsilon}^c=\gamma_\alpha \cap B(w_+(a),\varepsilon)^c$ where $B(a,R)$ represents a ball of radius $R$ with center $a$.}
This means that we have
	\begin{equation}\label{propproof:Gsplit}
		\begin{split}
			&\overline{G}(n,[\lambda n^{\frac{1}{3}}] +r,[\alpha n-\xi c_0 n^{\frac{1}{3}}])\\&
			= \frac{(-1)^{n+[\lambda n^{\frac{1}{3}}] +r+1}}{2\pi \mathrm{i}} \int_{\mathcal{G}_{w_+(\alpha),\varepsilon}} \text{d}w \; A_n(w,\alpha,\lambda,c_0\xi,r)e^{n\big(\tilde{S}_n(w,\alpha,\lambda,c_0\xi,r)+O\big(\frac{1}{n^2}\big)\big)} \\
			&+\frac{(-1)^{n+[\lambda n^{\frac{1}{3}}] +r+1}}{2\pi \mathrm{i}} \int_{\mathcal{G}_{w_+(\alpha),\varepsilon}^c} \text{d}w \; A_n(w,\alpha,\lambda,c_0\xi,r)e^{n\big(\tilde{S}_n(w,\alpha,\lambda,c_0\xi,r)+O\big(\frac{1}{n^2}\big)\big)}.
		\end{split}
	\end{equation}
	Since we have established that $\gamma_\alpha$ is the contour of steepest descent, we have that {for $t\in \gamma_\alpha$ with $t\not=w_+(\alpha)$}
	\begin{equation}
		\mathrm{Re}[{S_1(t,\alpha)}]<-C \mathrm{dist}_{t \in \gamma_\alpha} (t,w_+(\alpha)),
	\end{equation}
	for some constant $C>0$.  Standard saddle point approximations show that the contribution from the second term in the right side of~\eqref{propproof:Gsplit} decays exponentially in $n$ compared to the first term.  

	To approximate the first term on the right side 
of~\eqref{propproof:Gsplit}, we apply a local change of variables $w\mapsto w_+(\alpha)+c_1 n^{-\frac{1}{3}} \omega$ where $c_1$ is defined in~\eqref{eq:scalings}.  Using Taylor's series on each of the terms $S_1,S_2, S_3$ and $A_n$ defined in~\eqref{eq:S1}, \eqref{eq:S2}, \eqref{eq:S3} and~\eqref{eq:An}, we obtain after a computation
	\begin{equation}
		\begin{split}
			&A_n(w_+(\alpha)+c_1 n^{-\frac{1}{3}}\omega,\alpha,\lambda,c_0\xi,r)e^{n (\tilde{S}_n(w_+(\alpha)+c_1 n^{-\frac{1}{3}}\omega,\alpha,\lambda,c_0\xi,r)+O(\frac{1}{n^2}))}  \\
			&= \frac{3 \sqrt{3}}{2}(1+O(n^{-\frac{1}{3}})) e^{\frac{\omega^3}{3} -\xi \omega -9c_1 \lambda \omega  +\mathrm{Err}},
		\end{split}
	\end{equation}
	where $\mathrm{Err}=n^{-\frac{1}{3}}(O(\omega^4)+O(\omega^2)+O(\omega))$ is the error from applying the Taylor series expansion.  Under the change of variables $w\mapsto w_+(\alpha)+c_1 n^{-\frac{1}{3}} \omega$, the contour $\mathcal{G}_{w_+(\alpha),\varepsilon}$ is mapped to the contour $-\Gamma_{\mathrm{Ai},n}^{\varepsilon}$ where $\Gamma_{\mathrm{Ai},n}^{\varepsilon}=B(0,n^{\frac{1}{3}} \varepsilon) \cap \Gamma_{\mathrm{Ai}}$.  Here, the negative sign means that the contour is negatively oriented, that is oriented from $n^{\frac{1}{3}} \varepsilon e^{\frac{\mathrm{i} \pi}{3}}$ to $n^{\frac{1}{3}} \varepsilon e^{-\frac{\mathrm{i} \pi}{3}}$, which is due to the fact that $\gamma_\alpha$ is positively oriented.  Reversing the orientation, applying the change of variables $w\mapsto w_+(\alpha)+c_1 n^{-\frac{1}{3}} \omega$ and Taylor approximation to the first integral on the right side of~\eqref{propproof:Gsplit} gives that~\eqref{propproof:Gsplit} is equal to 
	\begin{equation}\label{propproof:Gsplitfirst1}
		\frac{(-1)^{n+[\lambda n^{\frac{1}{3}}] +r}}{2\pi \mathrm{i}} \int_{\Gamma_{\mathrm{Ai},n}^{\varepsilon}} \text{d} \omega \; \frac{3 \sqrt{3}c_1}{2n^{\frac{1}{3}}}(1+O(n^{-\frac{1}{3}})) e^{\frac{\omega^3}{3} -\xi \omega -9c_1 \lambda \omega  +\mathrm{Err}}.
	\end{equation}
	We choose $\varepsilon$ small enough so that the error terms are not significant. Indeed, we can control the error term using the $|e^y-1|\leq |y|e^{y}$ where $y$ represents the error term and so we need to bound
	\begin{equation}
		\frac{1}{2 \pi} \int_{\Gamma_{\mathrm{Ai},n}^{\varepsilon}} \text{d} \omega \; \frac{3 \sqrt{3}c_1}{2n^{\frac{1}{3}}} |\mathrm{Err}| |e^{\frac{\omega^3}{3} -\xi \omega -9c_1 \lambda \omega }||e^{\mathrm{Err}}|.
	\end{equation}
	For $n$ large, the $e^{\omega^3/3}$ term dominates the integral at the ends of the contour $\Gamma_{\mathrm{Ai},n}^{\varepsilon}$, and so the above integral is bounded by $n^{-\frac{2}{3}}$ as $n\to \infty$.  Thus,~\eqref{propproof:Gsplitfirst1} is equal to 
	\begin{equation}
		\frac{(-1)^{n+[\lambda n^{\frac{1}{3}}] +r}(1+O(n^{-\frac{1}{3}})}{2\pi \mathrm{i}} \int_{\Gamma_{\mathrm{Ai},n}^{\varepsilon}} \text{d} \omega \; \frac{3 \sqrt{3}c_1}{2n^{\frac{1}{3}}}(1+O(n^{-\frac{1}{3}})) e^{\frac{\omega^3}{3} -\xi \omega -9c_1 \lambda \omega }.
	\end{equation}
	We can extend the contour to $\Gamma_{\mathrm{Ai}}$ which introduces an exponentially small error (in terms of $n$). This verifies~\eqref{propproof:overlineGasymp}.  We apply the same steps as given for~\eqref{propproof:overlineGasymp} to obtain~\eqref{propproof:overlineHasymp} using the mapping $\xi\mapsto \xi+\mu/c_0=\xi+3 \sqrt{3}\mu  c_1 $ and noticing that 
	\begin{equation}
		\frac{n-nw+1}{n+nw-\alpha n-\xi n^{\frac{1}{3}}}=2(1+O(n^{-\frac{1}{3}}))
	\end{equation}
	under the change of variables $w=w_+(\alpha)+c_1n^{-\frac{1}{3}}\omega$.  
\end{proof}

	We are now in the position to evaluate the asymptotics of $f^{1,1}$,$f^{2,2}$ and $f^{1,2}$. 
Write
	\begin{equation}\label{propproof:eq:scalings}
x=\alpha n -n^{\frac{1}{3}}c_1\xi, \quad y=\alpha n -n^{\frac{1}{3}}c_1\eta.
	\end{equation}
 We start with $f^{1,1}$.
From~\eqref{eq:f11} we have
	\begin{equation}
	\label{propproof:eqf11}
		\begin{split}
			(c_0 n^{\frac{1}{3}})^2 f^{1,1}([x] ,[y]) 
=& (c_0 n^{\frac{1}{3}})^2 \sum_{\ell=0}^\kappa 
\Big( \overline{G}(n,\ell+1,[x]) \overline{G}(n,\ell,[y])\\
&-\overline{G}(n,\ell,[x]) \overline{G}(n,\ell+1,[y]) \Big).
		\end{split}
	\end{equation}
\cref{lem:estimates} and \cref{lem:changeofvariables} means that we can restrict the sum over $\ell$ to $n^{1/3+\delta}$. More precisely, for $\ell >n^{\frac{1}{3}+\delta}$, we have from \cref{lem:changeofvariables}
\begin{equation}
	\overline{G}(n,\ell,[x])=\overline{G}(n-\ell,0,[x]-2\ell)
\end{equation}
and  that
\begin{equation}
	x-2\ell= \alpha(n-\ell) -(2-\alpha)\ell -n^{\frac{1}{3}} c_0 \xi.
\end{equation}
This means we can use \cref{lem:estimates} with $m=n-\ell$ and $v=(2-\alpha)\ell m^{-\frac{1}{3}} +n^{\frac{1}{3}} m^{-\frac{1}{3}}c_0 \xi \geq C n^{\delta}$ for some constant $C>0$ since $\alpha<2$. 
Using~\eqref{propproof:overlineGasymp} with both $r=0$ and $r=1$ and using that the Riemann sum converges to an integral means that {the left hand side of} ~\eqref{propproof:eqf11} is equal to 
\begin{equation}
	\begin{split}
		- & \frac{(1+O(n^{-\frac{1}{3}}))}{(2\pi \mathrm{i})^2} \int_0^\infty \text{d} \lambda \int_{\Gamma_{\mathrm{Ai}}} \text{d} \omega_1 \int_{\Gamma_{\mathrm{Ai}}}\text{d} \omega_2 \; \left( \frac{3 \sqrt{3}c_0c_1}{2} \right)^2n^{\frac{1}{3}} \\
		&\times  \big((1-9c_1 \omega_1n^{-\frac{1}{3}}) - (1-9 c_1 \omega_2n^{-\frac{1}{3}})\big) e^{\frac{\omega_1^3}{3} + \frac{\omega_2^3}{3} -\xi \omega_1 -9c_1 \lambda \omega_1 - \eta \omega_2 -9c_1 \omega_2 \lambda} \\
		= & \frac{(1+O(n^{-\frac{1}{3}}))} {(2\pi \mathrm{i})^2}\int_0^\infty \text{d} \lambda \int_{\Gamma_{\mathrm{Ai}}} \text{d} \omega_1 \int_{\Gamma_{\mathrm{Ai}}}\text{d} \omega_2 \; \left( \frac{3 \sqrt{3}c_0c_1}{2} \right)^2  (9c_1 \omega_1 - 9 c_1 \omega_2)\\ 
		&\times e^{\frac{\omega_1^3}{3} + \frac{\omega_2^3}{3} -\xi \omega_1 -9c_1 \lambda \omega_1 - \eta \omega_2 -9c_1 \omega_2 \lambda} \\
		= & \frac{(1+O(n^{-\frac{1}{3}}))}{4} \int_0^\infty \text{d} \lambda \int_{\Gamma_{\mathrm{Ai}}} \text{d} \omega_1 \int_{\Gamma_{\mathrm{Ai}}}\text{d} \omega_2  \, ( \omega_1 -  \omega_2) e^{\frac{\omega_1^3}{3} + \frac{\omega_2^3}{3} -\xi \omega_1 - \lambda \omega_1 - \eta \omega_2 - \omega_2 \lambda},
	\end{split}
\end{equation}
as $c_1 c_0=1/(3 \sqrt{3})$ and using the change of variables $\lambda \mapsto \frac{\lambda}{9c_1}$. One sees that  this is exactly $K_{\mathrm{GOE}}^{11}(\xi,\eta)$ as given in~\eqref{eq:KGOE11}.

Next, we evaluate $f^{2,2}([x],[y])$.  We have from~\eqref{eq:f22} and~\eqref{propproof:eq:scalings}
\begin{equation}\label{propproof:eq:f22eqn1}
\begin{split}
    f^{2,2}&([x],[y]) = \mathrm{sgn}([x]-[y]) + \sum_{\ell=0}^\kappa \sum_{m_1=0}^{[x]} \sum_{m_2=0}^{[y]} \Big(  \overline{H}(n,\ell+1,[x]-m_1) 
    \\& \times \overline{H}(n,\ell,[y]-m_2) 
    - \overline{H}(n,\ell,[x]-m_1) \overline{H}(n,\ell+1,[y]-m_2) \Big) 
\\&-(-1)^n\sum_{m=0}^{[x]} \overline{H}(n,0,[x]-m)+(-1)^n\sum_{m=0}^{[y]} \overline{H}(n,0,[y]-m).
\end{split}
\end{equation}
We have that $\mathrm{sgn}([x]-[y])=-\mathrm{sgn}(\xi-\eta)$.  By \cref{lem:estimates} and \cref{lem:changeofvariables}, we can restrict each of the sums to $n^{\frac{1}{3}+\delta}$, use the asymptotic formulas for $\overline{H}$ given in~\eqref{propproof:overlineHasymp} and use the Riemann sum convergence. We obtain that~\eqref{propproof:eq:f22eqn1} equals 
\begin{equation}
\begin{split}
&-\mathrm{sgn}(\xi-\eta)- \int_0^\infty \text{d}\lambda \int_0^\infty \text{d}\mu_1\int_0^\infty \text{d}\mu_2 \; \frac{ (1+O(n^{-\frac{1}{3}}))n}{(2\pi \mathrm{i})^2} \int_{\Gamma_{\mathrm{Ai}}} \text{d} \omega_1 \int_{\Gamma_{\mathrm{Ai}}} \text{d} \omega_2 \\
&\times \frac{27c_1^2}{n^\frac{2}{3}} \big((1-9c_1 \omega_1n^{-\frac{1}{3}}) - (1-9 c_1 \omega_2n^{-\frac{1}{3}})\big) e^{\frac{\omega_1^3}{3} + \frac{\omega_2^3}{3} - \omega_1(\xi +9c_1 \lambda +3 \sqrt{3} c_1 \mu)  }\\
&\times e^{- \omega_2(\eta +9c_1 \lambda +3 \sqrt{3} c_1 \mu)} -\int_0^{\infty} \text{d}\mu \; \frac{(1+O(n^{-\frac{1}{3}}))}{2\pi \mathrm{i}} \int_{\Gamma_\mathrm{Ai}} d \omega 3 \; \sqrt{3} c_1e^{\frac{\omega^3}{3} -\omega(\xi +3 \sqrt{3}c_1 \mu)}\\&+\int_0^{\infty} \text{d}\mu \; \frac{(1+O(n^{-\frac{1}{3}}))}{2\pi \mathrm{i}} \int_{\Gamma_\mathrm{Ai}} d \omega \; 3 \sqrt{3} c_1e^{\frac{\omega^3}{3} -\omega(\eta+3 \sqrt{3}c_1 \mu)}.
\end{split}
\end{equation}
We change variables $\mu_i \mapsto \frac{\mu_i}{3\sqrt{3}c_1} $ for $i \in \{1,2\}$ (and  $\mu \mapsto \frac{\mu}{3\sqrt{3}c_1} $) as well as $\lambda \mapsto \frac{\lambda}{9c_1}$. The above equation is then equal to,  up to an error of $O(n^{-\frac{1}{3}})$ which we will ignore,
\begin{equation}
\begin{split}
&-\mathrm{sgn}(\xi-\eta)- \int_0^\infty \text{d}\lambda \int_0^\infty \text{d}\mu_1\int_0^\infty \text{d}\mu_2 \; \frac{ 1}{(2\pi \mathrm{i})^2} \int_{\Gamma_{\mathrm{Ai}}} \text{d} \omega_1 \int_{\Gamma_{\mathrm{Ai}}} \text{d} \omega_2 \; (\omega_1-\omega_2) \\
&\times   e^{\frac{\omega_1^3}{3} + \frac{\omega_2^3}{3} - \omega_1(\xi + \lambda + \mu) - \omega_2(\eta + \lambda +\mu) }-\int_0^{\infty}\text{d}\mu \; \mathrm{Ai}(\xi+\mu) +\int_0^{\infty}\text{d}\mu \; \mathrm{Ai}(\eta+\mu) .
\end{split}
\end{equation}
For the second term in the above equation, we split into two terms and integrate each of these terms with respect to $\mu_1$ and $\mu_2$ respectively. We obtain that the above equation is equal to 
\begin{equation}
\begin{split}
&-\mathrm{sgn}(\xi-\eta)- \int_0^\infty \text{d}\lambda \; \int_0^\infty \text{d}\mu \; \mathrm{Ai}(\xi+\lambda)\mathrm{Ai}(\eta+\lambda+\mu)
-\int_0^{\infty}\text{d}\mu \; \mathrm{Ai}(\xi+\mu) \\
&+\int_0^\infty \text{d}\lambda \int_0^\infty \text{d}\mu \; \mathrm{Ai}(\xi+\lambda+\mu)\mathrm{Ai}(\eta+\lambda)
+\int_0^{\infty}\text{d}\mu \; \mathrm{Ai}(\eta+\mu).
\end{split}
\end{equation}
This is exactly $K_{\mathrm{GOE}}^{22}(\xi,\eta)$ as given in~\eqref{eq:KGOE22} after a change of variables in the second and third terms.  

Finally, we consider the $f^{1,2}$ term. From~\eqref{eq:f12}, we need to consider
\begin{equation}\label{propproof:eq:f12eqn1}
\begin{split}
    c_0 n^{\frac{1}{3}} & f^{1,2}([x],[y]) =  - c_0n^{\frac{1}{3}}\sum_{\ell=0}^\kappa \sum_{m=0}^y \big(  \overline{G}(n,\ell+1,[x]) \overline{H}(n,\ell,[y]-m)
\\ & { 
    -\overline{G}(n,\ell,[x]) \overline{H}(n,\ell+1,[y]-m)\big)+ c_0n^{\frac{1}{3}} (-1)^n \overline{G}(n,0,[x])},
\end{split}
\end{equation}
with $x$ and $y$ given in~\eqref{propproof:eq:scalings}. 
By \cref{lem:estimates} and \cref{lem:changeofvariables}, we can restrict each of the sums to $n^{\frac{1}{3}+\delta}$, use the asymptotic formulas for $\overline{G}$ and  $\overline{H}$ given in~\eqref{propproof:overlineGasymp} and ~\eqref{propproof:overlineHasymp} and use the Riemann sum convergence. We obtain that~\eqref{propproof:eq:f12eqn1} equals, up to an error of $O(n^{-\frac{1}{3}})$ which we will ignore, 
\begin{equation}
\begin{split}
&\frac{1}{2 \pi \mathrm{i}} \int_{\Gamma_{\mathrm{Ai}}} \text{d} \omega \;\frac{ 3 \sqrt{3}c_1 c_0}{2} e^{\frac{\omega^3}{3} - \xi \omega} + \int_0^\infty \text{d} \mu \int_0^\infty \text{d} \lambda \; \frac{1}{(2 \pi \mathrm{i})^2}
\int_{\Gamma_{\mathrm{Ai}}} \text{d} \omega_1 \int_{\Gamma_{\mathrm{Ai}}} \text{d} \omega_2 \\
&\times \left( c_1 3 \sqrt{3} \right)^2 c_0 n^{\frac{1}{3}} \big((1-9c_1 \omega_1n^{-\frac{1}{3}}) - (1-9 c_1 \omega_2n^{-\frac{1}{3}})\big) \\
&\times  e^{\frac{\omega_1^3}{3} + \frac{\omega_2^3}{3} - \omega_1(\xi +9c_1 \lambda ) - \omega_2(\eta +9c_1 \lambda +3 \sqrt{3} c_1 \mu) }.
\end{split}
\end{equation}
For the first term in the above equation, we use~\eqref{eq:scalings} to simplify, while for the second term, we use the change of variables $\lambda \mapsto \frac{\lambda}{9c_1}$ and $\mu \mapsto \frac{\mu}{3 \sqrt{3}c_1}$. The above equation then becomes
\begin{equation}
\begin{split}
&\frac{1}{2} \mathrm{Ai}(\xi)-\int_0^\infty \text{d} \mu \int_0^\infty \text{d} \lambda \; \frac{1}{(2 \pi \mathrm{i})^2} 
\int_{\Gamma_{\mathrm{Ai}}} \text{d} \omega_1 \int_{\Gamma_{\mathrm{Ai}}} \text{d} \omega_2 \; \\
&\times (\omega_1-\omega_2)  e^{\frac{\omega_1^3}{3}+ \frac{\omega_2^3}{3}-\omega_1(\xi+\lambda) - \omega_2(\eta+\lambda+\mu)}\\
&=\frac{1}{2} \mathrm{Ai}(\xi)-\int_0^\infty \text{d} \mu \int_\mu^\infty \text{d} \lambda \; \frac{1}{(2 \pi \mathrm{i})^2} 
\int_{\Gamma_{\mathrm{Ai}}} \text{d} \omega_1 \int_{\Gamma_{\mathrm{Ai}}} \text{d} \omega_2 \;  \\
&\times \omega_1 e^{\frac{\omega_1^3}{3}+ \frac{\omega_2^3}{3}-\omega_1(\xi+\lambda-\mu) - \omega_2(\eta+\lambda)}\\
&+\int_0^\infty \text{d} \mu \int_0^\infty \text{d} \lambda \; \frac{1}{(2 \pi \mathrm{i})^2} \int_{\Gamma_{\mathrm{Ai}}} \text{d} \omega_1 \int_{\Gamma_{\mathrm{Ai}}} \text{d} \omega_2 \; \omega_2 e^{\frac{\omega_1^3}{3}+ \frac{\omega_2^3}{3}-\omega_1(\xi+\lambda) - \omega_2(\eta+\lambda+\mu)},
\end{split}
\end{equation}
where we have used the change of variables $\lambda \mapsto \lambda-\mu$ in the second integral. 
For the second term in the above equation, we swap the integrals with respect to $\mu$ and $\lambda$ and integrate with respect to $\mu$. For the third term, we integrate with respect to $\mu$. The above equation is then equal to 
\begin{multline}
\frac{1}{2} \mathrm{Ai}(\xi) - \frac{1}{2} \mathrm{Ai}(\xi) \int_0^{\infty} \text{d} \lambda \; \mathrm{Ai}(\eta+\lambda) \\
+ \frac{1}{2}  \int_0^{\infty} \text{d} \lambda \; \mathrm{Ai}(\xi+\lambda)\mathrm{Ai}(\eta+\lambda)+\frac{1}{2}  \int_0^{\infty} \text{d} \lambda \; \mathrm{Ai}(\xi+\lambda)\mathrm{Ai}(\eta+\lambda).
\end{multline}
This is equal to~\eqref{eq:KGOE12} since
\begin{align*}
1-\int_0^\infty \text{d} \lambda \; \mathrm{Ai}(\eta+\lambda)=& \int_{-\infty}^\infty \text{d} \lambda \; \mathrm{Ai}(\eta+\lambda)-\int_0^\infty \text{d} \lambda \; \mathrm{Ai}(\eta+\lambda) \\
=& \int_0^\infty \text{d} \lambda \; \mathrm{Ai}(\eta-\lambda).
\end{align*}
This completes the proof.
\end{proof}

\begin{proof}[Proof of \cref{lem:estimates}]
	The bounds for $|v| \leq M$ are a consequence of~\eqref{propproof:overlineGasymp} and~\eqref{propproof:overlineHasymp}.  We only need to consider the bounds for $v>M$ and for $\overline{G}$.  We set $a=\alpha-v m^{-\frac{2}{3}}$.
From \cref{lem:startingasymp}, we must bound
\begin{equation}\label{lemproof:estimates:start}
\frac{1}{2\pi \mathrm{i}} \int_{\gamma_a} \text{d} w \; A_m(w,a,0,0,0)e^{m(\tilde{S}_m(w,a,0,0,0)+O(\frac{1}{m^2}))}.
\end{equation} 
The deformation of the contour given in \cref{lem:startingasymp} to $\gamma_a$ does not pick up any additional contributions since no poles are passed under this deformation.  
As in the proof of~\eqref{propproof:overlineGasymp}, we split the contour into two parts, $\gamma_a \cap B(w_+(a),\varepsilon)$ and $\gamma_a \cap B(w_+(a),\varepsilon)^c$.  The same reasoning as given in the proof of~\eqref{propproof:overlineGasymp} means we only need to consider the contribution from $\gamma_a \cap B(w_+(a),\varepsilon)$.  

The local angle of the steepest descent contour at $w_+(a)$ means that we take the change of variables $w=w_+(a)+\mathrm{i} \omega m^{-\frac{1}{2}}$ with $\omega \in (-m^{\frac{1}{2}} \varepsilon, m^{\frac{1}{2}} \varepsilon)$ for $\varepsilon>0$.  Under the change of variables, the main contribution of~\eqref{lemproof:estimates:start} is equal to 
\begin{multline}
\frac{1}{2\pi} \int_{-m^{\frac{1}{2}} \varepsilon}^{m^{\frac{1}{2}} \varepsilon} \text{d} \omega \; m^{-\frac{1}{2}} A_m(w_+(a)+m^{-\frac{1}{2}} \mathrm{i} \omega,a,0,0,0) \\
\times e^{m (\tilde{S}_m(w_+(a)+m^{-\frac{1}{2}} \mathrm{i} \omega,a,0,0,0)+O(\frac{1}{m^2}))}.
\end{multline}
We have that 
$$
 A_m(w_+(a)+m^{-\frac{1}{2}} \mathrm{i} \omega,a,0,0,0)=d(a)m(1+O(m^{-\frac{1}{2}})),
$$
where $d(a)$ is explicit with $0<d(a)<\infty$ for $0<a<\alpha$.  We also have that 
\begin{equation}
\begin{split}
&\tilde{S}_m(w_+(a)+m^{-\frac{1}{2}} \mathrm{i} \omega,a,0,0,0)=S_1(w_+(a)+m^{-\frac{1}{2}} \mathrm{i} \omega,a)\\&+m^{-1} S_3(w_+(a)+m^{-\frac{1}{2}} \mathrm{i} \omega,a,0) - \frac{\log m}{m}+\frac{1}{m}.
\end{split}
\end{equation}
From here, we have that
\begin{equation}
S_1(w_+(a)+m^{-\frac{1}{2}} \mathrm{i} \omega,a)=S_1(w_+(a),a)-\frac{1}{2} S_1''(w_+(a),a)\omega^2+O(m^{-\frac{1}{2}} \omega^3),
\end{equation}
and 
\begin{equation}
m^{-1} S_3(w_+(a)+m^{-\frac{1}{2}} \mathrm{i} \omega,a,0)= m^{-1}S_3(w_+(a) \mathrm{i} \omega,a,0)+O(m^{-\frac{3}{2}}\omega).
\end{equation}
Note that $S_1''(w_+(a),a)>0$ by \cref{lem:D2S1}. 
Putting the above together, extending the region of integration and computing the Gaussian integral, we obtain a bound
\begin{equation}\label{lemproof:estimate:exponent}
\frac{C d(a)}{\sqrt{ m S_1'' (w_+(a),a)}} e^{m S_1(w_+(a),a)},
\end{equation}
where $C>0$. 
Now, a computation shows that $\frac{C d(a)}{\sqrt{ m S_1'' (w_+(a),a)}}$ is of the form $$\frac{\sqrt{\tilde{d}(a)}}{\sqrt{(1+(-4+a)a)m}},$$ where $0<\tilde{d}(a)<\infty$ is a rational function in $a$.  The function $\tilde{d}(a)$ can be computed explicitly, but its exact form is not important. Since $a \leq \alpha -M/m$, we conclude that $\frac{C d(a)}{\sqrt{ m S_1'' (w_+(a),a)}}$ is bounded by a constant that is dependent on $M$.  To bound the exponential term in~\eqref{lemproof:estimate:exponent}, introduce
\begin{equation}
g(a)=S_1(w_+(a),a).
\end{equation}
Then, since $S_1(w_+(\alpha),\alpha)=0$, we can write 
\begin{equation}
\begin{split}
g(a)&=S_1(w_+(a),a)-S_1(w_+(\alpha),\alpha)=\int_0^{a-\alpha}\text{d}s \; g'(\alpha+s)\\&=-\int_0^{\alpha-a}\text{d}s \; g'(\alpha-s).
\end{split}
\end{equation}
We also have that 
\begin{equation}
\frac{d}{da} g(a)=\frac{d}{da} S_1(w_+(a),a)=\frac{ \partial S_1}{\partial w}(w_+(a),a)w_+'(a)+\frac{ \partial S_1}{\partial a}(w_+(a),a).
\end{equation}
Notice that the first term on the right side is zero since $w=w_+(a)$ is a critical point for $S_1(w,a)$. Computing $\frac{ \partial S_1}{\partial a}$ and evaluating at $(w_+(a),a)$ gives
\begin{equation}
g'(a)=\log \bigg[ \frac{1-a+w_+(a)}{a-2w_+(a)} \bigg]=\log \bigg[-1-\frac{2}{\sqrt{1-4a+a^2}-1} \bigg]>0
\end{equation}
for $a<\alpha$.  We can conclude that for $v>M$, there exists some constant $c$ such that 
\begin{equation}
e^{m S_1(w_+(a),a)} \leq e^{-cv}
\end{equation}
as required.  The bounds for $\overline{H}$ follow from a similar computation. 
\end{proof}

Finally, we prove \cref{prop:kernelbounds}.

\begin{proof}[Proof of \cref{prop:kernelbounds}]
The proof of these estimates follow immediately from applying \cref{lem:changeofvariables} and \cref{lem:estimates}.  Note that $f^{2,2}$ is bounded by a constant due to the sign term. 
\end{proof}

\section{Point process formula derivation}  \label{sec:derivation}

\subsection{Inverse Kasteleyn matrix for {$G_n^{\mathbf{m}}$}} 

{In this section, we recall the Kasteleyn matrix for TSSCPPs and give a formula for its inverse from \cite{AC:20}.   
To avoid any confusion, we use the same notation used there. }
The importance of the inverse Kasteleyn matrix is that it gives a way to compute local statistics, which in our case, means that we can establish the correlations for the particle process introduced above. This will be shown the next subsection.

The entries of the skew-symmetric Kasteleyn matrix $K_n=K_n(x,y)_{x,y \in \mathtt{V}_n^{\mathbf{m}}}$ are given by
\begin{equation}\label{KasteleynMatrix}
	K_n((x_1,x_2),(y_1,y_2)) =k_n((x_1,x_2),(y_1,y_2))-k_n((y_1,y_2),(x_1,x_2)),
\end{equation}
where
\begin{equation}	
	k_n((x_1,x_2),(y_1,y_2)) = \begin{cases}
		1 & \mbox{if } [x_1+x_2]_2 =0 , x_2=y_2, x_1-y_1=1, \\
		1 & \mbox{if }  [x_1+x_2]_2=0,|y_2-x_2|=1, x_1=y_1,  \\
		1 & \mbox{if } x_1=x_2, y_1 = y_2 = x_1-1, \\
	0 & \mbox{otherwise.}  
	\end{cases}
	\end{equation}
Here the Kasteleyn orientation is chosen so that the number of counter-clockwise arrows around each face is an odd number. 
The edges along the diagonal are oriented towards the origin. From~\cite{Kas61,Kas63}, $|\mathrm{Pf} K_n|=Z_n$ is equal to the number of \tsscpp s  of size $n+1$. 

To {give formulas for} $K_n^{-1}$, define 
	\begin{equation}\label{eq:pnkl}
	p(n,k,\ell) = \frac{ (n+k-2\ell+1)!(2n-k-\ell+1)!}{(k-\ell)! (3n-k+2-2\ell)!} (-1)^k (3n-3k+2) C_{n-k},
\end{equation}
where $C_n = \frac{1}{n+1} \binom{2n}{n}$ is the $n$'th Catalan number. 
Introduce the following formulas for $0 \leq i \leq 2n-1$:
	\begin{equation}\label{hnb0}
		h_n^{0,\mathbf{b}}(i) =-[i+1]_2+ \sum_{k=0}^n p(n,k,0) \frac{1}{2 \pi \mathrm{i}} \int_{\Gamma_0} \text{d}r \; \frac{(1+r)^{n-k}}{(1-r)r^{i-2k}}, 
	\end{equation}
and
\begin{equation}\label{hnb1}
	h_n^{1,\mathbf{b}}(i) = \sum_{k=0}^n p(n,k,0) \frac{1}{2 \pi \mathrm{i}} \int_{\Gamma_0} \text{d}r \; \frac{(1+r)^{n-k}}{r^{i-2k+1}},
	\end{equation}
as well as the following formulas for $0 \leq i,j \leq 2n-1$:
\begin{equation}\label{tn00}
	\begin{split}
		&t_n^{0,0} (i,j)= \mathbbm{1}_{[i<j]}[i+1]_2[j]_2 - \mathbbm{1}_{[i>j]}[i]_2[j+1]_2 +[j+1]_2h_n^{0,\mathbf{b}}(i)\\
	 & -[i+1]_2 h_n^{0,\mathbf{b}}(j) + \sum_{k_1=0}^n \sum_{\ell_1=0}^{k_1} \sum_{k_2=0}^n \sum_{\ell_2=0}^{k_2} p(n,k_1,\ell_1)p(n,k_2,\ell_2) \frac{1}{(2\pi \mathrm{i})^4} \int_{\Gamma_0}\text{d}r_1  \\
&  \times \int_{\Gamma_0} \text{d}r_2 \int_{\Gamma_0}\text{d}s_1 \int_{\Gamma_0}\text{d}s_2 \; \frac{ (1+r_1)^{n-k_1} (1+r_2)^{n-k_2}}{(1-r_1)r_1^{i-2k_1}(1-r_2)r_2^{j-2k_2}}
		\frac{s_1-s_2}{(s_1 s_2-1)s_1^{\ell_1+1}s_2^{\ell_2+1}}, 
\end{split}
\end{equation}
\begin{equation}\label{tn10}
\begin{split}
	t_n^{1,0} (i,j)=& [j+1]_2h_n^{1,\mathbf{b}}(i)\\
& +	\sum_{k_1=0}^n \sum_{\ell_1=0}^{k_1} \sum_{k_2=0}^n \sum_{\ell_2=0}^{k_2} p(n,k_1,\ell_1)p(n,k_2,\ell_2) \frac{1}{(2\pi \mathrm{i})^4} \int_{\Gamma_0}\text{d}r_1 \int_{\Gamma_0} \text{d}r_2 \\
	& \times \int_{\Gamma_0}\text{d}s_1 \int_{\Gamma_0}\text{d}s_2 \; \frac{ (1+r_1)^{n-k_1} (1+r_2)^{n-k_2}}{r_1^{i-2k_1+1}(1-r_2)r_2^{j-2k_2}}	\frac{s_1-s_2}{(s_1 s_2-1)s_1^{\ell_1+1}s_2^{\ell_2+1}},
\end{split}
\end{equation}
\begin{equation}\label{tn01}
\begin{split}
	t_n^{0,1} (i,j)=& -[i+1]_2 h_n^{1,\mathbf{b}}(j)\\
	& +	\sum_{k_1=0}^n \sum_{\ell_1=0}^{k_1} \sum_{k_2=0}^n \sum_{\ell_2=0}^{k_2} p(n,k_1,\ell_1)p(n,k_2,\ell_2) \frac{1}{(2\pi \mathrm{i})^4} \int_{\Gamma_0}\text{d}r_1 \int_{\Gamma_0} \text{d}r_2 \\
& \times  \int_{\Gamma_0}\text{d}s_1 \int_{\Gamma_0}\text{d}s_2 \; \frac{ (1+r_1)^{n-k_1} (1+r_2)^{n-k_2}}{(1-r_1)r_1^{i-2k_1}r_2^{j-2k_2+1}} 	\frac{s_1-s_2}{(s_1 s_2-1)s_1^{\ell_1+1}s_2^{\ell_2+1}},
\end{split}
\end{equation}
and
\begin{equation}\label{tn11}
\begin{split}
	&t_n^{1,1} (i,j)=
	\sum_{k_1=0}^n \sum_{\ell_1=0}^{k_1} \sum_{k_2=0}^n \sum_{\ell_2=0}^{k_2} p(n,k_1,\ell_1)p(n,k_2,\ell_2) \frac{1}{(2\pi \mathrm{i})^4} \int_{\Gamma_0}\text{d}r_1 \int_{\Gamma_0} \text{d}r_2\\
& \times  \int_{\Gamma_0}\text{d}s_1 \int_{\Gamma_0}\text{d}s_2 \; \frac{ (1+r_1)^{n-k_1} (1+r_2)^{n-k_2}}{r_1^{i-2k_1+1}r_2^{j-2k_2+1}}  	\frac{s_1-s_2}{(s_1 s_2-1)s_1^{\ell_1+1}s_2^{\ell_2+1}}.
\end{split}
\end{equation}

\begin{thm}[{\cite[Theorem 3.2]{AC:20}}]
\label{thm:oldmain}
Suppose that $(x_1,x_2)=(i_1,i_1+2i_2+ \epsilon_i)$, $(y_1,y_2)=(j_1,j_1+2j_2+\epsilon_j)$ with $0\leq i_1 ,j_1 \leq  2n-1$
and $\epsilon_i, \epsilon_j \in \{0,1\}$, where $0 \leq i_2 \leq n- \lfloor (i_1 + \epsilon_i)/2 \rfloor$, and $0 \leq j_2 \leq n- \lfloor (j_1 + \epsilon_j)/2 \rfloor $. 
	
\begin{itemize}
	
\item If $\epsilon_i = \epsilon_j= 1$, then
	\begin{equation}\label{Kinv11}
		\begin{split}
			K_n^{-1}((x_1,x_2),(y_1,y_2)) =& \sum_{\ell_1=0}^{i_1} \sum_{\ell_2=0}^{j_1}  (-1)^{i_2+j_2} (-1)^{\ell_1+\ell_2} \\ & \times \binom{i_2-1+\ell_1}{\ell_1}\binom{j_2-1+\ell_2}{\ell_2}
 t_n^{1,1}(i_1-\ell_1,j_1-\ell_2).
		\end{split}
	\end{equation}

\item If $\epsilon_i = \epsilon_j= 0$, then
	\begin{equation}\label{Kinv00}
		\begin{split}
			K_n^{-1}((x_1,x_2),(y_1,y_2)) &= \sum_{\ell_1=0}^{i_2} \sum_{\ell_2=0}^{j_2}  (-1)^{i_2+j_2}  \binom{i_2}{\ell_1}\binom{ j_2}{\ell_2} t_n^{0,0}(i_1+\ell_1,j_1+\ell_2).
		\end{split}
	\end{equation}

\item If $\epsilon_i = 1$ and $\epsilon_j= 0$, then
	\begin{equation}\label{Kinv10}
		\begin{split}
			&K_n^{-1}((x_1,x_2),(y_1,y_2)) = \sum_{\ell_1=0}^{i_1} \sum_{\ell_2=0}^{j_2}  (-1)^{i_2+j_2} (-1)^{\ell_1} \binom{ i_2-1+\ell_1}{ \ell_1}\binom{j_2}{\ell_2}\\ & \times  t_n^{1,0}(i_1-\ell_1,j_1+\ell_2) - (-1)^{i_2+j_2} \mathbbm{1}_{x_1 \geq y_1} \mathbbm{1}_{x_1+x_2<y_1+y_2}  \binom{j_2-i_2-1}{ i_1-j_1}.
		\end{split}
	\end{equation}

\end{itemize}
\end{thm}

Having the formula for $K_n^{-1}$ means that we can compute local statistics {using this classical result}. 

\begin{thm}[{Montroll-Potts-Ward~\cite{MPW63}, Kenyon~\cite{Ken97}}]
\label{thm:localstats}
	The probability of edges $e_1=(v_1,v_2),\dots, e_m=(v_{2m-1},v_{2m})$ on $\mathtt{V}_n^{\mathtt{m}}$ is given by 
	\begin{equation}
		\mathbb{P}_n^{\mathbf{m}}[e_1,\dots, e_m] = \prod_{k=1}^m K_n(v_{2k-1},v_{2k}) \; \mathrm{Pf} \left( K^{-1}_n(v_i,v_j) \right)^T_{1 \leq i,j \leq 2m}.
	\end{equation}
\end{thm}

\subsection{Proof of \cref{prop:discrete}}
In this subsection, we prove \cref{prop:discrete}. We first need to show that the particle process on $\mathcal{L}_n$ is given by a Pfaffian point process. 
\begin{lem} \label{lem:discretefirststep}
The point process $\{z_i\}$ on $\mathcal{L}_n$ is a Pfaffian point process with correlation kernel $\mathbf{f}:\mathcal{L}_n \times \mathcal{L}_n \to \mathbb{R}$  given by
\begin{equation}
\mathbf{f}(x,y)=\left( \begin{array}{cc}
f^{1,1}(x,y) & f^{1,2}(x,y) \\
f^{2,1}(x,y) & f^{2,2}(x,y) \end{array} \right),
\end{equation}
for $x,y \in \mathcal{L}_n$ with
\begin{equation}
\begin{split}
f^{1,1}(x,y)&= -K_n^{-1}((x,x+1),(y,y+1)),\\
f^{2,2}(x,y)&= -K_n^{-1}((x,x+2),(y,y+2)),\\
f^{1,2}(x,y)&= \mathbbm{I}_{x=y}-K_n^{-1}((x,x+1),(y,y+2)),\\
f^{2,1}(x,y)&=-f^{1,2}(y,x).
\end{split}
\end{equation}

\end{lem}

\begin{proof}
For $1 \leq k \leq m$ and $1 \leq i_k \leq n$ with $i_k$'s distinct, let $v_{2k-1}=((i_k,i_k+1))$, $v_{2k}=((i_k,i_k+2))$, and $e_k=(v_{2k-1},v_{2k})$ with $1 \leq k \leq m$.  Then, by ~\cref{thm:localstats}
\begin{equation}
\mathbb{P}_n^{\mathbf{m}}[e_1,\dots, e_m] =  \mathrm{Pf} \left( K^{-1}_n(v_i,v_j) \right)^T_{1 \leq i,j \leq 2m}.
\end{equation}
since $K_n(v_{2k-1},v_{2k})=1$ for $1 \leq k \leq m$.  Recall that there is a particle at $k \in \mathcal{L}_n$ if and only if $((k,k+1),(k,k+2))$ is not covered by dimer, so we want $\mathbb{P}_n^{\mathbf{m}}[e_1^c,\dots, e_m^c]$ where $e_k^c$ means that the edge $e_k$ is not covered by a dimer.  By an inclusion-exclusion argument, 
\begin{equation}
\mathbb{P}_n^{\mathbf{m}}[e_1^c,\dots, e_m^c]= \mathrm{Pf} \left( J_{2m}- ((K^{-1}_n(v_i,v_j)) ^T_{1 \leq i,j \leq 2m}\right),
\end{equation}
where $J_{2m}=(\mathbbm{I}_{j=i+1}-\mathbbm{I}_{j=i-1})_{1 \leq i,j \leq 2m}$. The lemma follows by comparing the entries in the correlation kernel {and using that $K_n$ and $K_n^{-1}$ are skew-symmetric}.
\end{proof}

We will also use the following computation. 
\begin{lem} \label{lem:sintegral}
For $l_1,l_2 \in \mathbb{N}_0$, we have
\begin{equation}
\frac{1}{(2\pi \mathrm{i})^2} \int_{\Gamma_0} \text{d}s_1 \int_{\Gamma_0} \text{d}s_2 \; \frac{s_1-s_2}{(s_1s_2-1) s_1^{l_1+1}s_2^{l_2+1}} = \mathbbm{I}_{l_2=l_1+1}-\mathbbm{I}_{l_2=l_1-1}
\end{equation}
\end{lem}
\begin{proof}
We have 
\begin{equation}
\begin{split}
&\frac{1}{(2\pi \mathrm{i})^2} \int_{\Gamma_0} \text{d}s_1 \int_{\Gamma_0} \text{d}s_2 \; \frac{s_1-s_2}{(s_1s_2-1) s_1^{l_1+1}s_2^{l_2+1}} \\
&= \frac{1}{(2\pi \mathrm{i})^2} \int_{\Gamma_0} \text{d}s_2 \int_{\Gamma_0} \text{d}s_1 \; \frac{s_1-s_2}{(s_1-1/s_2) s_1^{l_1+1}s_2^{l_2+2}}  
=-\frac{1}{2 \pi \mathrm{i}} \int_{\Gamma_0} \text{d}s \; \frac{1/s -s}{s^{l_2+2-(l_1+1)}}\\
&=-\frac{1}{2 \pi \mathrm{i}} \int_{\Gamma_0} \text{d}s \; \left(\frac{1}{s^{l_2-l_1+2}} -\frac{1}{s^{l_2-l_1}} \right)
=-(\mathbbm{I}_{l_2=l_1-1}-\mathbbm{I}_{l_2=l_1+1})
\end{split}
\end{equation}
where we have pushed the contour with respect to $s_1$ through infinity to get the second equality. 
\end{proof}

\begin{proof}[Proof of \cref{prop:discrete}]
Using \cref{lem:discretefirststep}, we need to manipulate the expressions for $f^{1,1}(x,y),f^{1,2}(x,y)$, and $f^{2,2}(x,y)$ to get those given in~\eqref{eq:f11}, \eqref{eq:f12} and~\eqref{eq:f22} respectively.  We consider each in turn starting with $f^{1,1}(x,y)$. From \cref{lem:discretefirststep}, \cref{lem:sintegral} and \cref{thm:oldmain},
\begin{equation}
\begin{split}\label{eq:propdiscpf:f11firststep}
f^{1,1}(x,y) =& -K_n^{-1}((x,x+1),(y,y+1)) =-t_n^{1,1}(x,y)\\
=& -\sum_{k_1,k_2=0}^n \sum_{\ell_1=0}^{k_1} \sum_{\ell_2=0}^{k_2} p(n,k_1,\ell_1)p(n,k_2,\ell_2) \\
& \times (\mathbbm{I}_{l_2=l_1+1}-\mathbbm{I}_{l_2=l_1-1})\binom{n-k_1}{x-2k_1}\binom{n-k_2}{y-2k_2},
\end{split}
\end{equation}
where we have used
\begin{equation}
\binom{n-k}{x-2k}=\frac{1}{2 \pi \mathrm{i}} \int_{\Gamma_0} { \text{d}r} \; \frac{(1+r)^{n-k}}{r^{x-2k+1}}.
\end{equation}
We can then split the sums in~\eqref{eq:propdiscpf:f11firststep} to get
\begin{equation}\label{eq:propdiscpf:ftildedef}
	f^{1,1}(x,y)=\tilde{f}^{1,1}(x,y)-\tilde{f}^{1,1}(y,x)
\end{equation}
where
\begin{equation}
\tilde{f}^{1,1}(x,y)=\sum_{k_1,k_2=0}^n \sum_{\ell=0}^{(k_1-1)\wedge k_2} p(n,k_1,\ell+1)p(n,k_2,\ell)  \binom{n-k_1}{x-2k_1}\binom{n-k_2}{y-2k_2}.
\end{equation}
From the definition of $p(n,k,\ell)$ in \eqref{eq:pnkl}, $p(n,k,\ell)=0$ if $k < \ell< {\frac{n+k+1}{2}}$ since
\begin{equation}
\frac{(n+k-2\ell+1)!}{(k-\ell)!}= (k-\ell+1)\dots (k-\ell+n-\ell+1).
\end{equation}
Choose $\tilde{\kappa}$ to be such that 
\begin{equation}\label{eq:propdiscpf:kappatilde}
\min(k_1,k_2) \leq \tilde{\kappa} \leq \min\bigg(\frac{n+k_1+1}{2},\frac{n+k_2+1}{2} \bigg).
\end{equation}
Then, 
\begin{equation}\label{eq:propdiscpf:ftilde}
\tilde{f}^{1,1}(x,y)=\sum_{k_1,k_2=0}^n \sum_{\ell=0}^{\tilde{\kappa}} p(n,k_1,\ell+1)p(n,k_2,\ell)  \binom{n-k_1}{x-2k_1}\binom{n-k_2}{y-2k_2}.
\end{equation}
Notice that we have that $p(n,k,\ell)=(-1)^\ell p(n-\ell,k-\ell,0)$. Define
\begin{equation}
P(n,k,x)=p(n,k,0)\binom{n-k}{x-2k}.
\end{equation}
Then we have
\begin{equation}\label{eq:propdiscpf:Pn}
\begin{split}
(-1)^\ell P(n-\ell,k-\ell,x-2\ell)&=(-1)^\ell p(n-\ell,k-\ell,0)\binom{(n-\ell)-(k-\ell)}{(x-2\ell)-2(k-\ell)}\\&
=p(n,k,\ell)\binom{n-k}{x-2k}.
\end{split}
\end{equation}
Notice that the expression is equal to zero if $k>[\frac{x}{2}]$. 
Substituting~\eqref{eq:propdiscpf:Pn} back into~\eqref{eq:propdiscpf:ftilde} gives 
\begin{equation}
\begin{split}
\label{eq:propdiscpf:ftilde1}
\tilde{f}^{1,1}(x,y) =& -\sum_{k_1,k_2=0}^n \sum_{\ell=0}^{\tilde{\kappa}} P(n-(\ell+1),k_1-(\ell+1),x-2(\ell+1)) \\
& \times P(n-\ell,k_2-\ell,y-2\ell).
\end{split}
\end{equation}
We can choose $\tilde{\kappa}$ to be equal to $\kappa$ as constrained in 
{\eqref{eq:kappa}},
which is valid for the point process on $\mathcal{L}_n$.  Then, we can write
\begin{equation}
P(n,k,x)=\frac{(-1)^k}{k!(n-k)!} \frac{S(n,k,x)}{n-k+1}, 
\end{equation}
where $S(n,k,x)$ is defined in~\eqref{eq:Snwx}. The map $z \mapsto \frac{1}{\Gamma(z)}$ is an entire function and so $w \mapsto S(n,w,x)$ is analytic in $\mathbb{C} \backslash \{ -(n+j+2),-(2n+j+2),-(n+\frac{j+1}{2}); j\geq 0  \}$.  Notice that for a function $F:\mathbb{C} \to \mathbb{R}$ analytic in a neighborhood of $k$, we have
\begin{equation}\label{eq:sumtointegral}
\frac{1}{2\pi \mathrm{i}} \int_{\Gamma_k} \text{d}w \; \frac{F(w)}{\prod_{j=0}^n (w-j)}=\frac{(-1)^{n-k}F(k)}{k!(n-k)!}.
\end{equation}
Using the above fact, we have that
\begin{equation}
P(n,k,x)=\frac{(-1)^n}{2\pi \mathrm{i}} \int_{\Gamma_k} \text{d}w \; \frac{S(n,w,x)}{(n-w+1)\prod_{j=0}^n (w-j)},
\end{equation}
and so 
\begin{equation}
\begin{split} \label{eq:propdiscpf:Ptocontour}
P(n-\ell,k-\ell,x-2\ell)&=\frac{(-1)^{n-\ell}}{2\pi \mathrm{i}} \int_{\Gamma_{k-\ell}} \text{d}w \; \frac{S(n-\ell,w,x-2\ell)}{(n-\ell-w+1)\prod_{j=0}^{n-\ell} (w-j)} \\
&=\frac{(-1)^{n-\ell}}{2\pi \mathrm{i}} \int_{\Gamma_{k}} \text{d}w \; \frac{S(n-\ell,w-\ell,x-2\ell)}{(n-w+1)\prod_{j=\ell}^n (w-j)} \\
&=\frac{(-1)^{n-\ell}}{2\pi \mathrm{i}} \int_{\Gamma_{k}} \text{d}w \; G_{n,\ell,x}(w),
\end{split}
\end{equation}
where we used the change of variables $w \mapsto w-\ell$ and $G_{n,\ell,x}(w)$ is defined in~\eqref{eq:Gnlx}. 
From~\eqref{eq:propdiscpf:Ptocontour}, we also have
\begin{equation}
P(n-(\ell+1),k-(\ell+1),x-2(\ell+1))=-\frac{(-1)^{n-\ell}}{2\pi \mathrm{i}} \int_{\Gamma_{k}} \text{d}w \; G_{n,\ell+1,x}(w).
\end{equation}
{Recall that $\Gamma_{\{x\}}'=\Gamma_{0,1,\dots, \lfloor \frac{x}{2}\rfloor}$ as defined in \cref{sec:asymptotics}}.
Substituting the above equation and~\eqref{eq:propdiscpf:Ptocontour} into~\eqref{eq:propdiscpf:ftilde1}, we have 
\begin{equation} \label{eq:propdiscpf:f11tildelast}
\begin{split}
\tilde{f}^{1,1}(x,y)&=\sum_{k_1,k_2=0}^n \sum_{\ell=0}^{{\kappa}} \left(\frac{1}{2\pi \mathrm{i}} \int_{\Gamma_{k_1}} \text{d}w \; G_{n,\ell+1,x}(w)\right) \left(\frac{1}{2\pi \mathrm{i}} \int_{\Gamma_{k_2}} \text{d}w \; G_{n,\ell,y}(w) \right)\\
&=\sum_{\ell=0}^{\kappa} \left(\frac{1}{2\pi \mathrm{i}} \int_{\Gamma_{\{x\}}'} \text{d}w \; G_{n,\ell+1,x}(w)\right) \left(\frac{1}{2\pi \mathrm{i}} \int_{\Gamma_{\{y\}}'} \text{d}w \; G_{n,\ell,y}(w) \right),
\end{split}
\end{equation}
since in the first line, the sum with respect to $k_1$ and $k_2$ only depend on the residues in the first and second integrals respectively.  {Note that the contours can be restricted to $\Gamma_{\{x\}}'$ and $\Gamma_{\{y\}}'$ due to the poles of the integrands, see~\eqref{eq:Snwxpoles}.}
Substituting the above formula into~\eqref{eq:propdiscpf:ftildedef} and using~\eqref{eq:Goverline} gives~\eqref{eq:f11}.  

We next consider the term $f^{2,2}(x,y)$. We will not do the full computation, but highlight the main differences with with the computation of $f^{1,1}(x,y)$ given above.  We have, using \cref{lem:discretefirststep}, \cref{lem:sintegral} and \cref{thm:oldmain}, 
\begin{equation}
\begin{split}\label{eq:propdiscpf:f22firststep}
&f^{2,2}(x,y) = -K_n^{-1}((x,x+2),(y,y+2))\\
& =-(t_n^{0,0}(x,y)+t_n^{0,0}(x+1,y)+t_n^{0,0}(x,y+1)+t_n^{0,0}(x+1,y+1))\\
&=-\sum_{k_1,k_2=0}^n \sum_{\ell_1=0}^{k_1} \sum_{\ell_2=0}^{k_2} p(n,k_1,\ell_1)p(n,k_2,\ell_2)  (\mathbbm{I}_{l_2=l_1+1}-\mathbbm{I}_{l_2=l_1-1}) \\
&\times \frac{1}{(2 \pi \mathrm{i})^2} \int_{\Gamma_0} \text{d}r_1 \int_{\Gamma_0} \text{d}r_2 \; \frac{(1+r_1)^{n-k_1} (1+r_2)^{n-k_2}(1+r_1^{-1}+r_2^{-1}+r_1^{-1}r_2^{-1})}{(1-r_1)r_1^{x-2k_1}(1-r_2)r_2^{y-2k_2} } \\
& -\sum_{k=0}^n p(n,k,0) \frac{1}{2\pi \mathrm{i}} \int_{\Gamma_0} \text{d}r \; \frac{(1+r)^{n-k}(1+r^{-1})}{r^{x-2k}(1-r)} \\
& +\sum_{k=0}^n p(n,k,0) \frac{1}{2\pi \mathrm{i}} \int_{\Gamma_0} \text{d}r \; \frac{(1+r)^{n-k}(1+r^{-1})}{r^{y-2k}(1-r)} 
-(\mathbbm{I}_{x<y}-\mathbbm{I}_{x>y}) ,
\end{split}
\end{equation}
which follows by combining the terms and simplifying. Simplifying and using the definition of $\tilde{\kappa}$ provided in~\eqref{eq:propdiscpf:kappatilde} and following the steps to get to~\eqref{eq:propdiscpf:ftilde}, we obtain
\begin{equation}
\begin{split}\label{eq:propdiscpf:f22secondstep}
&f^{2,2}(x,y) \\
&=-\sum_{k_1,k_2=0}^n \sum_{\ell=0}^{\tilde{\kappa}} \big( p(n,k_1,\ell+1)p(n,k_2,\ell)-p(n,k_1,\ell+1)p(n,k_2,\ell+1)\big) \\
&\times \frac{1}{(2 \pi \mathrm{i})^2} \int_{\Gamma_0} \text{d}r_1 \int_{\Gamma_0} \text{d}r_2 \; \frac{(1+r_1)^{n-k_1+1} (1+r_2)^{n-k_2+1}}{(1-r_1)r_1^{x-2k_1+1}(1-r_2)r_2^{y-2k_2+1} } 
\\
& -\sum_{k=0}^n p(n,k,0) \frac{1}{2\pi \mathrm{i}} \int_{\Gamma_0} \text{d}r \; \frac{(1+r)^{n-k+1}}{r^{x-2k+1}(1-r)}\\
& +\sum_{k=0}^n p(n,k,0) \frac{1}{2\pi \mathrm{i}} \int_{\Gamma_0} \text{d}r \; \frac{(1+r)^{n-k+1}}{r^{y-2k+1}(1-r)}-(\mathbbm{I}_{x<y}-\mathbbm{I}_{x>y}).
\end{split}
\end{equation}
We now use the following identity:
\begin{equation}
\begin{split}
\frac{1}{2\pi \mathrm{i}} \int_{\Gamma_0} \text{d}r \; \frac{(1+r)^{n-k+1}}{r^{x-2k+1}(1-r)}= & \sum_{m=0}^{x-2k} \frac{1}{2\pi \mathrm{i}} \int_{\Gamma_0} \text{d}r \; \frac{(1+r)^{n-k+1}}{r^{x-2k+1-m}} \\
= & \sum_{m=0}^{x} \frac{1}{2\pi \mathrm{i}} \int_{\Gamma_0} \text{d}r \; \frac{(1+r)^{n-k+1}}{r^{x-2k+1-m}},
\end{split}
\end{equation}
where the last equality follows because there is no pole in the integrand when $m>x-2k$.  With this equation,~\eqref{eq:propdiscpf:f22secondstep} becomes
\begin{equation}
\begin{split}\label{eq:propdiscpf:f22thirdstep}
&f^{2,2}(x,y)=\sum_{k_1,k_2=0}^n \sum_{\ell=0}^{\kappa} \sum_{m_1=0}^x \sum_{m_2=0}^y 
\big( \tilde{P}_{\ell+1} (n,k_1,x-m_1)\tilde{P}_{\ell} (n,k_2,y-m_2)\\ 
&-\tilde{P}_{\ell} (n,k_1,x-m_1)\tilde{P}_{\ell+1} (n,k_2,y-m_2)\big)-(-1)^n\sum_{k=0}^n \sum_{m=0}^x \tilde{P}_0(n,k,x-m) \\
& +(-1)^n\sum_{k=0}^n \sum_{m=0}^y \tilde{P}_0(n,k,y-m)+\mathrm{sgn}(x-y),
\end{split}
\end{equation}
where 
\begin{equation}
\begin{split}
\tilde{P}_{\ell} (n,k,x)&=p(n-\ell,k-\ell,x-2\ell) \binom{n-k+1}{x-2k} \\
&=\frac{(-1)^{n-\ell}}{2 \pi \mathrm{i}} \int_{\Gamma_k} \text{d}w \; G_{n,\ell,x}(w) \frac{n-w+1}{n-x+w},
\end{split}
\end{equation}
{which follows using the same steps to obtain~\eqref{eq:propdiscpf:Ptocontour} from~\eqref{eq:propdiscpf:Pn}.} Note that the factors of $(-1)^n$ in front of the double sums in~\eqref{eq:propdiscpf:f22thirdstep} is from~\eqref{eq:sumtointegral}.  
Substituting the above equation back into~\eqref{eq:propdiscpf:f22thirdstep} and using the same step used to obtain~\eqref{eq:propdiscpf:f11tildelast} gives the formula for $f^{2,2}(x,y)$ found in~\eqref{eq:f22}.

We finally consider the term $f^{1,2}(x,y)$. We will again not perform the full computation, but extract the main formula and point to the computations given for $f^{1,1}(x,y)$ and $f^{2,2}(x,y)$ above {for the key ideas}.  We have, using \cref{lem:discretefirststep}, \cref{lem:sintegral} and \cref{thm:oldmain}, 
\begin{equation}
\begin{split}
&f^{1,2}(x,y)=\mathbbm{I}_{x=y} -K_n^{-1}((x,x+1),(y,y+2))\\
=&\mathbbm{I}_{x=y}+t_n^{1,0}(x,y)+t_n^{1,0}(x,y+1)-\mathbbm{I}_{x\geq y}\mathbbm{I}_{2x+1<2y+2}\binom{0}{x-y}\\
=&t_n^{1,0}(x,y)+t_n^{1,0}(x,y+1)\\
=&\sum_{k=0}^n p(n,k,0) \frac{1}{2\pi \mathrm{i}} \int_{\Gamma_0} \text{d} r \; \frac{(1+r)^{n-k}}{r^{x-2k+1}} +\sum_{k_1,k_2=0}^n \sum_{\ell_1=0}^{k_1} \sum_{\ell_2=0}^{k_2} p(n,k_1,\ell_1) p(n,k_2,\ell_2)  \\
&\times (\mathbbm{I}_{\ell_2=\ell_1+1} - \mathbbm{I}_{\ell_2=\ell_1-1})\frac{1}{(2\pi \mathrm{i})^2} \int_{\Gamma_0}\text{d}r_1 \int_{\Gamma_0} \text{d}r_2 \; \frac{(1+r_1)^{n-k_1}}{r_1^{x-2k_1+1}} \frac{(1+r_2)^{n-k_2+1}}{r_2^{y-2k_2+1}(1-r_2)} \\
=&\sum_{k=0}^n p(n,k,0) \frac{1}{2\pi \mathrm{i}} \int_{\Gamma_0} \text{d} r \; \frac{(1+r)^{n-k}}{r^{x-2k+1}} -\sum_{k_1,k_2=0}^n \sum_{\ell=0}^{\tilde{\kappa}} \big( p(n,k_1,\ell+1) p(n,k_2,\ell)\\
&-p(n,k_1,\ell) p(n,k_2,\ell+1) \big) \binom{n-k_1}{x-2k_1} \frac{1}{2\pi \mathrm{i}} \int_{\Gamma_0} \text{d}r \; \frac{(1+r)^{n-k_2+1}}{r^{x-2k_2+1}(1-r)},
\end{split}
\end{equation}
where $\tilde{\kappa}$ is defined in~\eqref{eq:propdiscpf:kappatilde}. We proceed mirroring the computations for  $f^{1,1}(x,y)$ and $f^{2,2}(x,y)$ above; we omit this part of the computation.  
\end{proof}

\bibliographystyle{alpha}
\bibliography{Biblio}

\end{document}